\documentclass[12pt,twoside]{article}
\usepackage[latin1]{inputenc}
\usepackage[english]{babel}
\usepackage{clm-macros}
\usepackage{amsmath}
\usepackage{psfrag}
\usepackage{epsfig}
\usepackage{graphicx}
\usepackage{verbatim}
\usepackage{tikz}
\usepackage{subfigure}

\usepackage{datetime}

\usepackage{fullpage}

\usepackage{enumerate}

\usetikzlibrary{decorations}
\usetikzlibrary{decorations.pathmorphing}
\usetikzlibrary{arrows}
\tikzstyle{every node}=[circle, draw, fill=white,inner sep=0pt, minimum width=4pt]
\tikzstyle{nodelabel}=[rounded corners,fill=none,inner sep=5pt,draw=none]
\tikzstyle{matching} = [ultra thick]
\tikzset{snake/.style={decorate, decoration=snake}}

\newdateformat{UKvardate}{
\THEDAY\ \monthname[\THEMONTH], \THEYEAR}

\def\Notname{Notation}
\def\casname{Case}
\def\cnjname{Conjecture}
\def\corname{Corollary}
\def\Defname{Definition}
\def\exename{Exercise}
\def\exaname{Example}
\def\lemname{Lemma}
\def\listassertionname{List of Assertions}
\def\prbname{Problem}
\def\prpname{Proposition}
\def\proofOfname{\proofname{} of}
\def\ssbname{Case}
\def\subname{Case}
\def\thmname{Theorem}

\begin{document}

\newcommand{\mcg}{matching covered graph}
\newcommand{\mc}{matching covered}

\title{Minimal Braces\footnote{This research is supported by
Fundect-MS, CNPq and
FAPESP (2018/04679-1) of Brazil, and by
the Austrian Science Foundation FWF (START grant Y463).}}
\author{Phelipe A. Fabres$^{a}$ \and
            Nishad Kothari$^{b}$ \and
            Marcelo H. de Carvalho$^{a}$ \\
{\small $^{a}$ FACOM --- UFMS, Campo Grande, Brasil}\\
{\small $^{b}$ IC --- UNICAMP, Campinas, Brasil}\\
}

%\footnotemark[\ref{note1}]

\UKvardate
\date{1 September, 2020}
  \maketitle 
  \thispagestyle{empty}

\begin{center}
{\Large In memory of Robin Thomas}
\end{center}

\begin{abstract}
McCuaig (2001, Brace Generation, J. Graph Theory 38: 124-169)
proved a generation theorem for braces, and used it as the principal induction tool
to obtain a structural characterization of Pfaffian braces
(2004, P{\'o}lya's Permanent Problem, Electronic J. Combinatorics 11: R79).

\smallskip
A brace is {\it minimal} if deleting any edge results in a graph that is not a brace.
From McCuaig's brace generation theorem, we derive our main theorem
that may be viewed as an induction tool for minimal braces.
As an application,
we prove that a minimal brace of order~$2n$ has size at most~$5n-10$, when $n \geq 6$,
and we provide a complete characterization of minimal braces that meet this
upper bound.

\smallskip
A similar work has already been done in the context of minimal bricks by
Norine and Thomas (2006, Minimal Bricks, J. Combin. Theory Ser. B 96: 505-513)
wherein they deduce the main result from the brick
generation theorem due to the same authors
(2007, Generating Bricks, J. Combin. Theory Ser. B 97: 769-817).
\end{abstract}

%\tableofcontents

%\newpage

\section{Bipartite matching covered graphs}
\label{sec:bip-mcgs}

For general graph-theoretic notation and terminology, we refer the reader
to Bondy and Murty \cite{bomu08}.
All graphs considered in this paper are finite and loopless; however, we do allow multiple (i.e., parallel) edges.
For a graph $G$, its {\it order} is the number of vertices (i.e., $|V(G)|$),
and its {\it size} is the number of edges (i.e., $|E(G)|$).
For a subset $X$ of $V(G)$, we denote by~$\partial(X)$
the cut associated with $X$, and we refer to $X$ and $\overline{X}:=V(G)-X$
as the {\it shores} of $\partial(X)$.
Thus $\partial(X)$ is the set of edges that have exactly one end in either shore.
A cut is {\it trivial} if either of its shores is a singleton.
The graph obtained by contracting the shore~$X$ to a single vertex~$x$
is denoted by $G/(X \rightarrow x)$, or simply by $G/X$.
The two graphs $G/X$ and $G/\overline{X}$ are called
the {\it $\partial(X)$-contractions} of $G$.

\smallskip
A connected graph~$G$ is {\it $k$-extendable} if it has a matching of cardinality $k$,
and if each such matching extends to (i.e., is a subset of) a perfect matching of~$G$.
For a comprehensive treatment of matching theory and its origins, we refer the
reader to Lov{\'a}sz and Plummer~\cite{lopl86}.
All graphs considered in this paper are $1$-extendable, and we shall
instead refer to them as {\it matching covered graphs}.
It is easily seen that these graphs (of order four or more) are \mbox{$2$-connected}.
Also, for a graph~$G$, we let $n_G:=\frac{|V(G)|}{2}$ and $m_G:=|E(G)|$;
whence $G$ has order $2 \cdot n_G$ and size $m_G$.

\smallskip
A cut $\partial(X)$ of $G$ is {\it tight} if $|M \cap \partial(X)|=1$ for each perfect matching $M$
of~$G$. A \mcg\ that is free of nontrivial tight cuts
is called a {\it brace} if it is bipartite, or otherwise a {\it brick}.
It is easily verified that if $\partial(X)$ is a nontrivial tight cut of a \mcg~$G$
then each $\partial(X)$-contraction of~$G$ is a \mcg\ of strictly smaller order.
This observation leads to a decomposition of any \mcg\ into a list
of bricks and braces; this procedure is known as a {\it tight cut decomposition} of $G$.
Clearly, a graph may admit several tight cut decompositions.
However, Lov{\'a}sz~\cite{lova87} proved the remarkable result that any two tight cut decompositions
of a \mcg~$G$ yield the same list of bricks and braces
(except possibly for the multiplicities of edges).
We remark that $G$ is bipartite if and only if its tight cut decomposition yields only braces
(i.e., it yields no bricks).
The bipartite \mcg~$Q_{10}$, shown in Figure~\ref{fig:Q_10}, has a nontrivial
tight cut $\partial(X)$, and each of its $\partial(X)$-contractions is isomorphic to the brace~$K_{3,3}$.

\begin{figure}[!htb]
\centering
\subfigure[ $\partial(X)$ is a nontrivial tight cut in $Q_{10}$]
{
\begin{tikzpicture}[scale=0.7]

\draw (2,3)node[nodelabel]{$X_-$};
\draw (2,-1)node[nodelabel]{$X_+$};
\draw (8,3)node[nodelabel]{$\overline{X}_+$};
\draw (8,-1)node[nodelabel]{$\overline{X}_-$};

\draw (0.6,2.4) -- (3.4,2.4) -- (3.4,1.6) -- (0.6,1.6) -- (0.6,2.4);

\draw (-0.4,0.4) -- (4.4,0.4) -- (4.4,-0.4) -- (-0.4,-0.4) -- (-0.4,0.4);

\draw (5.6,2.4) -- (10.4,2.4) -- (10.4,1.6) -- (5.6,1.6) -- (5.6,2.4);

\draw (6.6,0.4) -- (9.4,0.4) -- (9.4,-0.4) -- (6.6,-0.4) -- (6.6,0.4);

\draw (1,2) -- (0,0) -- (3,2) -- (2,0) -- (1,2) -- (4,0) -- (3,2);

\draw (7,0) -- (6,2) -- (9,0) -- (8,2) -- (7,0) -- (10,2) -- (9,0);

\draw (0,0) -- (6,2);
\draw (2,0) -- (8,2);
\draw (4,0) -- (10,2);

\draw (1,2)node{};
\draw (3,2)node{};
\draw (6,2)node{};
\draw (8,2)node{};
\draw (10,2)node{};

\draw (0,0)node[fill=black]{};
\draw (2,0)node[fill=black]{};
\draw (4,0)node[fill=black]{};
\draw (7,0)node[fill=black]{};
\draw (9,0)node[fill=black]{};

\end{tikzpicture}
\label{fig:Q_10}
}
\hspace*{0.2in}
\subfigure[Edge $e$ is not superfluous in $Q_{10}^+$]
{
\begin{tikzpicture}[scale=0.7]

\draw[ultra thick] (3,2) -- (7,0);
\draw (3.8,1.9)node[nodelabel]{$e$};

\draw (2,3)node[nodelabel]{$A_1$};
\draw (2,-1)node[nodelabel]{$B_1$};
\draw (8,3)node[nodelabel]{$A_2$};
\draw (8,-1)node[nodelabel]{$B_2$};

\draw (0.6,2.4) -- (3.4,2.4) -- (3.4,1.6) -- (0.6,1.6) -- (0.6,2.4);

\draw (-0.4,0.4) -- (4.4,0.4) -- (4.4,-0.4) -- (-0.4,-0.4) -- (-0.4,0.4);

\draw (5.6,2.4) -- (10.4,2.4) -- (10.4,1.6) -- (5.6,1.6) -- (5.6,2.4);

\draw (6.6,0.4) -- (9.4,0.4) -- (9.4,-0.4) -- (6.6,-0.4) -- (6.6,0.4);

\draw (1,2) -- (0,0) -- (3,2) -- (2,0) -- (1,2) -- (4,0) -- (3,2);

\draw (7,0) -- (6,2) -- (9,0) -- (8,2) -- (7,0) -- (10,2) -- (9,0);

\draw (0,0) -- (6,2);
\draw (2,0) -- (8,2);
\draw (4,0) -- (10,2);

\draw (1,2)node{};
\draw (3,2)node{};
\draw (6,2)node{};
\draw (8,2)node{};
\draw (10,2)node{};

\draw (0,0)node[fill=black]{};
\draw (2,0)node[fill=black]{};
\draw (4,0)node[fill=black]{};
\draw (7,0)node[fill=black]{};
\draw (9,0)node[fill=black]{};

\end{tikzpicture}
\label{fig:Q_10^+with-non-superfluous-edge}
}
\caption{The graphs $Q_{10}$ and $Q_{10}^+:=Q_{10}+e$}
\label{fig:tight-cut-and-non-superfluous-edge}
\end{figure}
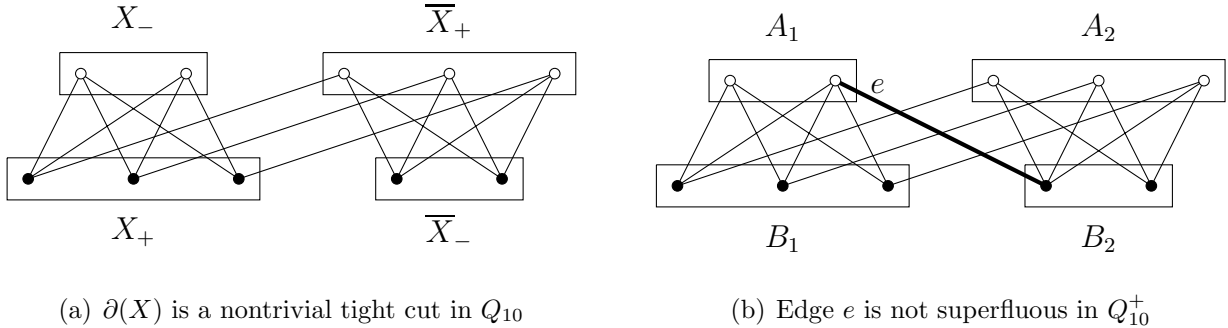

\smallskip
Several important properties of a \mcg~$G$ may be deduced by analysing its bricks and braces.
(For instance, $G$ is Pfaffian if and only if each of its bricks and braces is Pfaffian;
see \cite{vaya89,lire91}.) Consequently, researchers were led to
gain a deeper understanding of bricks and braces. McCuaig \cite{mccu01} established a generation
theorem for simple braces, and used this as the principal induction tool to obtain
a structural characterization of Pfaffian braces \cite{mccu04}.
Robertson, Seymour and Thomas~\cite{rst99} arrived at the same characterization
using a different approach. (These groundbreaking works led to a polynomial-time algorithm for deciding whether
or not a given bipartite graph is Pfaffian; see \cite{mrst97}.)

\smallskip
A brace is {\it minimal} if deleting any edge results in a graph that is not a brace.
The aforementioned McCuaig's Theorem is a powerful induction tool for the class of simple braces.
The object of this paper is to use McCuaig's Theorem to derive an induction
tool for the class of minimal braces --- a proper subset of the class of simple braces.
In the following three subsections, we introduce the necessary terminology to make this more precise.

\smallskip
A similar work has already been done in the context of minimal bricks by
Norine and Thomas \cite{noth06},
wherein they deduce the main result from the brick
generation theorem due to the same authors \cite{noth07}.

\subsection{Braces}
\label{sec:braces}

For a connected bipartite graph~$G$, we adopt the notation~$G[A,B]$ to denote its color classes.
We will generally use letters $A$ and $B$ to denote the color classes; sometimes we may instead use $A'$
and $B'$.
As shown in Figure~\ref{fig:minimality-preserving-pair}, members of $A$ (or of~$A'$) will be denoted using
letters $a$~and~$w$ (with subscripts and/or superscripts) and will be depicted using hollow nodes;
likewise, members of $B$ (or of~$B'$) will be denoted using letters $b$~and~$u$
(with subscripts and/or superscripts)
and will be depicted using solid nodes.

\smallskip
The neighborhood of a set of vertices $Z$ is denoted by~$N_G(Z)$.
The following may be deduced from the well-known Hall's Theorem.

\begin{prp}
\label{prp:bip-mcg-characterization}
For a connected bipartite graph $G[A,B]$, where \mbox{$|A| = |B|$},
the following are equivalent:
\begin{enumerate}[(i)]
\item $G$ is matching covered,
\item $|N_G(Z)| \geq |Z|+1$ for every nonempty proper subset $Z$ of $A$, and
\item $G-a-b$ has a perfect matching for each pair of vertices $a \in A$ and $b \in B$. \qed
\end{enumerate}
\end{prp}

Suppose that $X$ is an odd subset of the vertex set of a connected bipartite graph $G[A,B]$.
Then one of the two sets $A \cap X$ and $B \cap X$ is larger than the other;
the larger set, denoted~$X_+$, is called the {\it majority part} of~$X$;
the smaller set, denoted~$X_-$, is called the {\it minority part} of~$X$.
The following proposition provides a convenient way of visualizing
tight cuts in bipartite graphs. It is easily proved.
(See Figure~\ref{fig:Q_10} for an example.)

\begin{prp}
\label{prp:tight-cuts-in-bip-mcgs}
A cut $\partial(X)$ of a bipartite \mcg~$G[A,B]$ is tight if and only if the following hold:
\begin{enumerate}[(i)]
\item $|X|$ is odd and $|X_+|=|X_-|+1$; consequently $|\overline{X}_+|=|\overline{X}_-|+1$; and
\item there are no edges between $X_-$ and $\overline{X}_-$. \qed
\end{enumerate}
\end{prp}

Recall that a brace is a bipartite matching covered graph that is free of nontrivial tight cuts.
The following characterization of braces may be deduced from
Proposition~\ref{prp:tight-cuts-in-bip-mcgs}.

\begin{prp}
\label{prp:brace-characterization}
For a connected bipartite graph $G[A,B]$ of order six or more, where $|A|=|B|$,
the following are equivalent:
\begin{enumerate}[(i)]
\item $G$ is a brace,
\item $|N_G(Z)| \geq |Z|+2$ for every nonempty subset $Z$ of $A$ such that $|Z| < |A|-1$,
\item $G-a_1-a_2-b_1-b_2$ has a perfect matching for any four distinct vertices $a_1,a_2 \in A$
and $b_1,b_2 \in B$,
\item $G$ is $2$-extendable. \qed
\end{enumerate}
\end{prp}

Thus braces (of order at least four) are precisely those bipartite graphs that are $2$-extendable.
The following immediate consequence of
Propositions~\ref{prp:bip-mcg-characterization}~and~\ref{prp:brace-characterization}
is worth noting.

\begin{cor}
\label{cor:adding-an-edge}
Let $G$ denote a bipartite graph obtained from a graph~$H$ by adding an edge.
If $H$ is \mc\ then so is $G$.
Furthermore, if $H$ is a brace then so is $G$. \qed
\end{cor}

The braces $K_2$ and $C_4$ are the only simple bipartite {\mcg}s of order at most four.
For a bipartite \mcg~$G$ of order six or more, it is easy to show that if $G$ has a $2$-vertex-cut
then $G$ has a nontrivial tight cut. This implies the following.

\begin{prp}
Every brace, of order six or more, is $3$-connected. \qed
\end{prp}

A vertex is {\it cubic} if its degree equals three; it is {\it noncubic}
if it has degree four or more. A graph is {\it cubic} if each of its vertices is cubic,
and it is {\it noncubic} if it has a noncubic vertex.

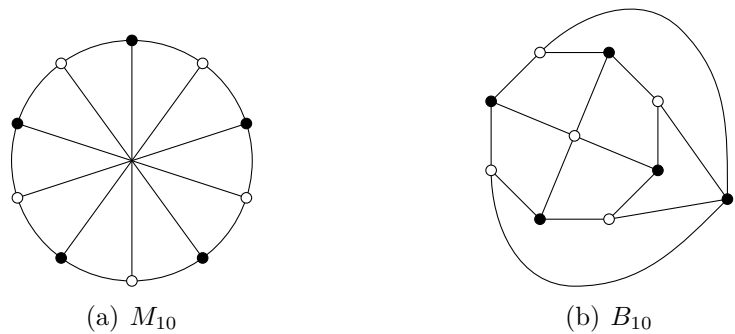
\begin{figure}[!htb]
\centering
\subfigure[$M_{10}$]{
\begin{tikzpicture}[scale=0.8]

\draw (0:0)circle(2);
\draw (90:2)node[fill=black]{} -- (270:2)node{};
\draw (126:2)node{} -- (306:2)node[fill=black]{};
\draw (162:2)node[fill=black]{} -- (342:2)node{};
\draw (198:2)node{} -- (18:2)node[fill=black]{};
\draw (234:2)node[fill=black]{} -- (54:2)node{};

\end{tikzpicture}
\label{fig:M_10}
}
\hspace*{1in}
\subfigure[$B_{10}$]{
\begin{tikzpicture}[scale=0.8]

\draw (337.5:2.75) to [out=90,in=315] (45:2.5) to [out=135,in=45] (112.5:1.5);
\draw (337.5:2.75) to [out=225,in=0] (270:2.5) to [out=180,in=270] (202.5:1.5);

\draw (337.5:2.75) -- (292.5:1.5);
\draw (337.5:2.75) -- (22.5:1.5);

\draw (0:0) -- (67.5:1.5);
\draw (0:0) -- (157.5:1.5);
\draw (0:0) -- (247.5:1.5);
\draw (0:0) -- (337.5:1.5);

\draw (22.5:1.5) -- (67.5:1.5) -- (112.5:1.5) -- (157.5:1.5) -- (202.5:1.5) -- (247.5:1.5) -- (292.5:1.5) -- (337.5:1.5) -- (22.5:1.5);

\draw (22.5:1.5)node{};
\draw (112.5:1.5)node{};
\draw (202.5:1.5)node{};
\draw (292.5:1.5)node{};
\draw (0:0)node{};

\draw (67.5:1.5)node[fill=black]{};
\draw (157.5:1.5)node[fill=black]{};
\draw (247.5:1.5)node[fill=black]{};
\draw (337.5:1.5)node[fill=black]{};
\draw (337.5:2.75)node[fill=black]{};

\end{tikzpicture}
\label{fig:B_10}
}
\caption{McCuaig braces of order ten}
\label{fig:Mccuaig-braces-order-ten}
\end{figure}

\smallskip
McCuaig \cite{mccu01} described three infinite families of simple braces:
prisms\footnote{McCuaig~\cite{mccu01} refers to `prisms' as `ladders'.},
M{\"o}bius ladders and biwheels.
A {\it biwheel} of order $2n$ (where $n \geq 4$), denoted $B_{2n}$,
is the simple bipartite graph obtained from the cycle graph~$C_{2n-2}$ by
adding two nonadjacent vertices --- each of which has degree
exactly $n-1$.
%Figure~\ref{fig:} shows the three smallest biwheels.
Observe that $B_{2n}$ has size $4n-4$.
The cube~$B_8$, shown in Figure~\ref{fig:B_8}, is the smallest biwheel.
Except for $B_8$, biwheels are noncubic; see Figure~\ref{fig:B_10}.

\smallskip
On the other hand, prisms and M{\"o}bius ladders are cubic;
we refer the interested reader to~\cite{clm15,mccu01} for descriptions of these families.
The cube~$B_8$ is the smallest prism.
The smallest M{\"o}bius ladders are $K_{3,3}$ and the brace~$M_{10}$
shown in Figure~\ref{fig:M_10}.
A {\it McCuaig brace} is any brace that is either a prism, or a M{\"o}bius ladder,
or a biwheel.

\subsection{McCuaig's Theorem}
\label{sec:McCuaig}

An edge~$e$ of a \mcg~$G$ is {\it removable} if $G-e$ is also \mc.
The following is easily deduced from
Propositions~\ref{prp:bip-mcg-characterization}~and~\ref{prp:brace-characterization}.

\begin{cor}
In a brace, of order six or more, every edge is removable. \qed
\end{cor}

Now let $G$ denote a brace of order six or more, and let $e \in E(G)$.
The bipartite \mcg~$G-e$ may not be a brace.
In particular,
one or both ends of $e$ may have degree precisely two (in $G-e$);
in order to recover a smaller brace, at the very least, we must get rid of vertices of degree two.
This brings us to the following notions of `bicontraction' and `retract'.

\smallskip
Let $G$ be a \mcg,
and let $v_0$ denote a vertex of degree two
that has two distinct neighbors, say $v_1$ and $v_2$.
The {\it bicontraction} of $v_0$ is the operation of contracting the two edges $v_0v_1$ and $v_0v_2$
incident with $v_0$. Note that $\partial(X)$, where $X:=\{v_0,v_1,v_2\}$,
is a tight cut of $G$. The graph obtained by bicontracting $v_0$ is the same as $G/X$
and is thus \mc. However, the bicontraction of a vertex of degree two in a simple
graph need not result in a simple graph. The {\it retract} of $G$, denoted~$\widehat{G}$,
is the \mcg\ obtained by bicontracting all its vertices of degree two that have two distinct neighbors.

\smallskip
For a brace~$G$ of order six or more,
a (removable) edge $e$ is {\it thin} if $\widehat{G-e}$ is also a brace.
Recently, Carvalho, Lucchesi and Murty \cite{clm15}
proved the following.

\begin{thm}
\label{thm:two-thin-edges-in-braces}
Every brace, of order six or more, has at least two thin edges.
\end{thm}

Note that if $e$ is a thin edge of a simple brace~$G$, the brace $\widehat{G-e}$ may not be simple.
A thin edge $e$ of a simple brace~$G$ is {\it strictly thin} if the brace $\widehat{G-e}$
is also simple.
For instance, every edge of $K_{3,3}$, and of $B_8$, is thin but none of them is strictly thin.
It is easily verified that every McCuaig brace
has several thin edges; however, none of them is strictly thin.
McCuaig showed that these are in fact the only simple braces with this property.
We let $\mathcal{G}$~denote the set that comprises $K_2$,~$C_4$ and all McCuaig braces.
McCuaig's Theorem \cite{mccu01}
may now be stated as follows.

\begin{thm}
{\sc [McCuaig's Theorem]}
\label{thm:strictly-thin-edge-in-simple-braces}
Every simple brace $G \notin \mathcal{G}$ has a strictly thin edge.
\end{thm}

\smallskip
Carvalho, Lucchesi and Murty gave an alternative proof of
McCuaig's Theorem using the existence
of a thin edge; see~\cite{clm08}.
In~\cite{clm15}, the same authors establish
a stronger version of McCuaig's Theorem.

\smallskip
For a strictly thin edge $e$ of a simple brace~$G$, the {\it index} of $e$,
denoted ${\rm index}(e)$, is the number of vertices of degree two in $G-e$.
Clearly, ${\rm index}(e) \in \{0,1,2\}$, depending on how many ends of $e$
are cubic in~$G$. The following is easily verified; see Figures~\ref{fig:index-one}~and~\ref{fig:index-two}.

\begin{prp}
\label{prp:compare-order-size-of-G-and-H}
Let $e$ denote a strictly thin edge of a simple brace~$G$, and let $H:=\widehat{G-e}$.
Then $n_H = n_G - {\rm index}(e)$ and $m_H = m_G - 1 - 2\cdot{\rm index}(e)$. \qed
\end{prp}

\subsection{Minimal Braces}
\label{sec:minimal-braces}

Recall that a brace $G$ is minimal if, for each $e \in E(G)$, the graph~$G-e$ is not a brace.
An edge~$e$ of a simple brace~$G$ is {\it superfluous} if $G-e$ is also a brace;
note that a `superfluous edge' is the same as a `strictly thin edge of index zero'.
Thus, a minimal brace is a brace devoid of superfluous edges.
Since any superfluous edge must join two noncubic vertices, the following holds.

\begin{prp}
\label{prp:brace-noncubic-stable-implies-minimal}
Let $G$ denote a brace of order six or more. If the set of all noncubic vertices is a stable set then $G$ is a
minimal brace. \qed
\end{prp}

The graph~$B_8^+$ shown in Figure~\ref{fig:B_8^+},
obtained from~$B_8$ by adding an edge, is the smallest simple
brace that is not minimal; it has a unique superfluous edge.
On the other hand, the graph~$Q_{10}^+$ shown in Figure~\ref{fig:Q_10^+with-non-superfluous-edge},
obtained from~$Q_{10}$ by adding an edge, is minimal.
%(In particular, the converse of Proposition~\ref{prp:brace-noncubic-stable-implies-minimal}
%does not hold.)

\smallskip
As stated earlier, our main objective is to derive an induction tool for the class of minimal braces
from McCuaig's Theorem.

\smallskip
Now let $G$ denote a minimal brace that is not a member of $\mathcal{G}$.
By McCuaig's Theorem, $G$~has a strictly thin edge, say~$e$.
We let $H$ denote the simple brace~$\widehat{G-e}$.
However, $H$ may not be a minimal brace.
Clearly, we may choose a set $F \subset E(H)$ such that $J:=H-F$
is a minimal brace. (In this manner, we may recover a smaller minimal brace~$J$.)
Note that ${\rm index}(e) \in \{1,2\}$, each member of $F$ is a superfluous edge of $H$,
and that $F = \emptyset$ if and only if $H$ is a minimal brace.
%Also, since each edge of~$H$ corresponds naturally to an edge of~$G$,
%the set~$F$ may instead be viewed as a subset of $E(G)$.
This brings us to the following definition.

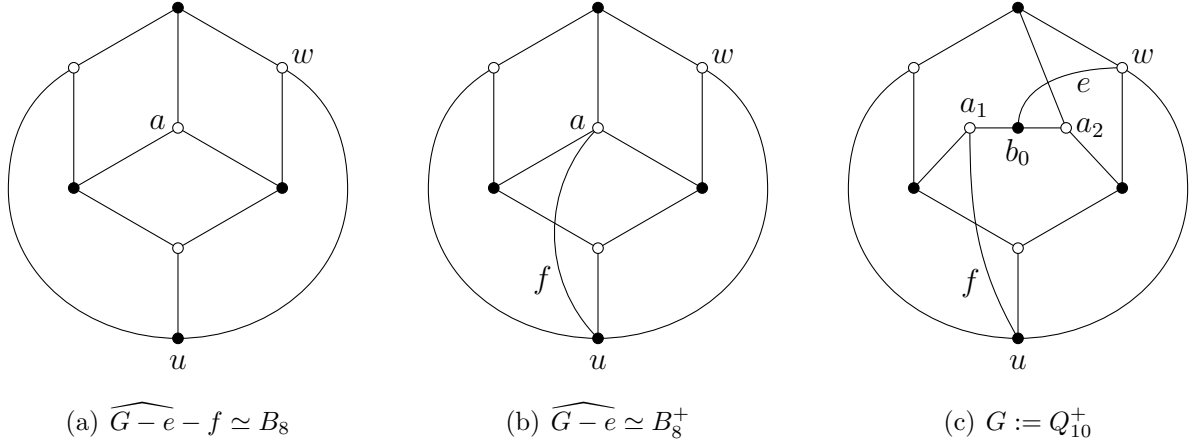
\begin{figure}[!htb]
\centering
\subfigure[$\widehat{G-e}-f \simeq B_8$]
{
\begin{tikzpicture}[scale=0.8]

\draw (90:2) -- (150:2) -- (210:2) -- (270:2) -- (330:2) -- (30:2) -- (90:2);

\draw (0:0) -- (90:2);
\draw (0:0) -- (210:2);
\draw (0:0) -- (330:2);

\draw (270:3.5) -- (270:2);
\draw (270:3.5) to [out=180,in=270] (200:3) to [out=90,in=210] (150:2);
\draw (270:3.5) to [out=0,in=270] (340:3) to [out=90,in=330] (30:2);

\draw (90:2)node[fill=black]{};
\draw (210:2)node[fill=black]{};
\draw (330:2)node[fill=black]{};
\draw (270:3.5)node[fill=black]{};

\draw (150:2)node{};
\draw (270:2)node{};
\draw (30:2)node{};
\draw (0:0)node{};

\draw (30:2.4)node[nodelabel]{$w$};
\draw (270:3.9)node[nodelabel]{$u$};
\draw (165:0.35)node[nodelabel]{$a$};

\end{tikzpicture}
\label{fig:B_8}
}
\hspace*{0.2in}
\subfigure[$\widehat{G-e} \simeq B_8^+$]
{
\begin{tikzpicture}[scale=0.8]

\draw (270:3.5) to [out=135,in=225] (0:0);

\draw (90:2) -- (150:2) -- (210:2) -- (270:2) -- (330:2) -- (30:2) -- (90:2);

\draw (0:0) -- (90:2);
\draw (0:0) -- (210:2);
\draw (0:0) -- (330:2);

\draw (270:3.5) -- (270:2);
\draw (270:3.5) to [out=180,in=270] (200:3) to [out=90,in=210] (150:2);
\draw (270:3.5) to [out=0,in=270] (340:3) to [out=90,in=330] (30:2);

\draw (90:2)node[fill=black]{};
\draw (210:2)node[fill=black]{};
\draw (330:2)node[fill=black]{};
\draw (270:3.5)node[fill=black]{};

\draw (150:2)node{};
\draw (270:2)node{};
\draw (30:2)node{};
\draw (0:0)node{};

\draw (30:2.4)node[nodelabel]{$w$};
\draw (270:3.9)node[nodelabel]{$u$};
\draw (165:0.35)node[nodelabel]{$a$};

\draw (250:2.7)node[nodelabel]{$f$};

\end{tikzpicture}
\label{fig:B_8^+}
}
\hspace*{0.2in}
\subfigure[$G:=Q_{10}^+$]
{
\begin{tikzpicture}[scale=0.8]

\draw (0:0) to [out=90,in=180] (30:2);

\draw (0:0.8) -- (90:2);
\draw (0:0.8) -- (330:2);
\draw (0:-0.8) -- (210:2);
\draw (0:-0.8) to [out=270,in=120] (270:3.5);

\draw (0:0.8) -- (0:-0.8);

\draw (90:2) -- (150:2) -- (210:2) -- (270:2) -- (330:2) -- (30:2) -- (90:2);

%\draw (0:0) -- (90:2);
%\draw (0:0) -- (210:2);
%\draw (0:0) -- (330:2);

\draw (270:3.5) -- (270:2);
\draw (270:3.5) to [out=180,in=270] (200:3) to [out=90,in=210] (150:2);
\draw (270:3.5) to [out=0,in=270] (340:3) to [out=90,in=330] (30:2);

\draw (90:2)node[fill=black]{};
\draw (210:2)node[fill=black]{};
\draw (330:2)node[fill=black]{};
\draw (270:3.5)node[fill=black]{};
\draw (0:0)node[fill=black]{};

\draw (150:2)node{};
\draw (270:2)node{};
\draw (30:2)node{};
\draw (0:0.8)node{};
\draw (0:-0.8)node{};
%\draw (0:0)node[fill=black]{};

\draw (270:0.4)node[nodelabel]{$b_0$};
\draw (155:0.8)node[nodelabel]{$a_1$};
\draw (0:1.2)node[nodelabel]{$a_2$};
\draw (30:2.4)node[nodelabel]{$w$};
\draw (270:3.9)node[nodelabel]{$u$};

\draw (33:1.3)node[nodelabel]{$e$};
\draw (253:2.7)node[nodelabel]{$f$};

\end{tikzpicture}
\label{fig:Q_10^+}
}
\caption{$(e,\{f\})$ is a minimality-preserving pair of the minimal brace~$Q_{10}^+$}
\label{fig:minimality-preserving-pair}
\end{figure}

\begin{Def}
\label{Def:minimality-preserving-pair}
{\sc [Minimality-Preserving Pair]}
For a minimal brace $G$,
a pair $(e,F)$ is a minimality-preserving pair if $e$ is a strictly thin edge of~$G$,
and $F$ is a subset of
$E(\widehat{G-e})$ so that the graph~$\widehat{G-e}-F$
is a minimal brace.
\end{Def}

In the above definition, since each edge of $\widehat{G-e}$ naturally corresponds
to an edge of~$G$, one may instead view the set $F$ as a subset of~$E(G)$.
Figure~\ref{fig:minimality-preserving-pair} shows an example of a
minimal brace~$G:=Q_{10}^+$ with a minimality-preserving pair $(e,\{f\})$
where ${\rm index}(e)=1$.

\smallskip
The following is a trivial consequence of McCuaig's Theorem.

\begin{cor}
\label{cor:minimality-preserving-pair-in-minimal-braces}
Every minimal brace $G \notin \mathcal{G}$ has a minimality-preserving pair $(e,F)$
for any strictly thin edge~$e$. \qed
\end{cor}

The next statement follows immediately from Proposition~\ref{prp:compare-order-size-of-G-and-H}.

\begin{prp}
\label{prp:compare-order-size-of-G-and-J}
Let $(e,F)$ denote a minimality-preserving pair of a minimal brace~$G$, and
let $J:=\widehat{G-e}-F$.
Then $n_J = n_G - {\rm index}(e)$ and $m_J = m_G - 1 - 2 \cdot {\rm index}(e) - |F|$. \qed
\end{prp}

Corollary~\ref{cor:minimality-preserving-pair-in-minimal-braces}
may be viewed as an induction tool for minimal braces;
however, it is not particularly useful for the following reason.
If $(e,F)$
is a minimality-preserving pair of a minimal brace $G$,
then the minimal brace~$J:=\widehat{G-e}-F$
can be arbitrarily smaller in size than $G$ depending on the cardinality of the
set~$F$.
(This is in contrast to McCuaig's Theorem; see Proposition~\ref{prp:compare-order-size-of-G-and-H}).
On the other hand, it seems intuitive that for a minimal brace $G \notin \mathcal{G}$
one should be able to find a minimality-preserving pair $(e,F)$
such that the set~$F$ is ``small''. This is in fact a consequence of our
Main Theorem~(\ref{thm:narrow-minimality-preserving-pair-in-minimal-braces}).

\begin{cor}
\label{cor:narrow-minimality-preserving-pair-in-minimal-braces-abridged-version}
Every minimal brace $G \notin \mathcal{G}$
has a minimality-preserving pair $(e,F)$ such that $|F| \leq {\rm index}(e) + 1$.
\end{cor}

Apart from this quantitative information regarding the minimality-preserving pair~$(e,F)$,
our Main Theorem~(\ref{thm:narrow-minimality-preserving-pair-in-minimal-braces})
also provides qualitative information: for instance, each member of~$F$
is at distance~one from the strictly thin edge~$e$.
%\newpage

\bigskip
\noindent
{\bf Organization of this paper:}
The Main Theorem (\ref{thm:narrow-minimality-preserving-pair-in-minimal-braces})
and its proof appear in Section~\ref{sec:induction-tool-minimal-braces}.
In Section~\ref{sec:application}, we use the Main Theorem as an induction tool
to prove Theorem~\ref{thm:extremal-minimal-braces} ---
which states that $m_G \leq 5n_G-10$ for any minimal brace~$G$, where $n_G \geq 6$,
and also provides a complete characterization of minimal braces that meet this upper bound.
In Section~\ref{sec:small-order}, we characterize minimal braces of small order;
this will serve as the base case in our proof of Theorem~\ref{thm:extremal-minimal-braces}.

\section{Minimal braces of small order}
\label{sec:small-order}

Let $e$ denote any (removable) edge of a simple brace~$G$ of order six or more.
Note that $e$ is not superfluous if and only if
the bipartite \mcg~$G-e$ has a nontrivial tight cut.
One may now easily deduce the following from
Proposition~\ref{prp:tight-cuts-in-bip-mcgs}.

\begin{cor}
\label{cor:non-superfluous-edges-in-braces}
Let $G[A,B]$ denote a simple brace of order six or more.
An edge~$e$ of $G$ is not superfluous if and only if there exist partitions
$(A_1,A_2)$ of $A$ and $(B_1,B_2)$ of $B$ such that $|B_1| = |A_1| + 1$
and $e$ is the only edge joining a vertex in~$A_1$ to a vertex in~$B_2$. \qed
\end{cor}

The brace~$Q_{10}^+$, shown in Figure~\ref{fig:Q_10^+with-non-superfluous-edge},
has precisely one edge joining two noncubic vertices, and it is not superfluous;
consequently, $Q_{10}^+$ is a minimal brace.

\smallskip
For a simple bipartite graph~$G[A,B]$, its {\it bipartite
complement}~$\overline{G}$ is the graph that has the same set of vertices
and has edge set $\{ab: ab \notin E(G), a \in A, b \in B\}$.
Using this notion, one may easily prove that $K_{3,3}$~and~$B_8$ are the only
simple cubic bipartite graphs of order at most eight.
(This proof technique is illustrated in the proof of~Proposition~\ref{prp:cubic-bipartite-graphs-order-ten}.)
Consequently, $K_{3,3}$ is the only simple brace of order six.
Using Corollary~\ref{cor:non-superfluous-edges-in-braces},
one may infer that $B_8$ is the only minimal brace of order eight.

\begin{prp}
\label{prp:minimal-braces-order-at-most-eight}
The only minimal braces of order at most eight are $K_2, C_4, K_{3,3}$ and $B_8$.~\qed
\end{prp}

\begin{prp}
\label{prp:cubic-bipartite-graphs-order-ten}
The only simple cubic bipartite graphs of order ten are $M_{10}$ and $Q_{10}$.
\end{prp}
\begin{proof}
Let $G[A,B]$ denote a simple cubic bipartite graph of order~ten,
and let $\overline{G}$ denote its bipartite complement.
Clearly, $\overline{G}$ is $2$-regular.
Observe the following.
If $\overline{G}$ is disconnected then $G$ is isomorphic to $Q_{10}$.
Otherwise $\overline{G}$ is connected and $G$ is isomorphic to~$M_{10}$.
\end{proof}

\begin{prp}
\label{prp:minimal-braces-order-ten}
The only minimal braces of order ten are $M_{10}$, $B_{10}$ and $Q_{10}^+$.
\end{prp}
\begin{proof}
Let $G[A,B]$ denote a minimal brace of order ten. If $G$ is cubic then,
by Proposition~\ref{prp:cubic-bipartite-graphs-order-ten},
$G$ is isomorphic to $M_{10}$. Now assume that $G$ is noncubic,
and let $T$ denote the set of noncubic vertices.

\smallskip
First suppose that $T$ is a stable set.
Observe that $T$ has precisely two members, say $a$~and~$b$, each of which
has degree precisely four.  Consequently, $G-a-b$ is a connected
$2$-regular bipartite graph. (Recall that braces of order six or more are $3$-connected.)
Thus $G-a-b$ is isomorphic to $C_8$ and $G$ is isomorphic to $B_{10}$.

\smallskip
Now suppose that $T$ is not a stable set, and let $e:=ab$ denote an edge that joins $a \in A \cap T$
and $b \in B \cap T$. Since $G$ is devoid of superfluous edges, the edge~$e$ in particular
is not superfluous. By Corollary~\ref{cor:non-superfluous-edges-in-braces},
there exist partitions $(A_1,A_2)$ of~$A$ and $(B_1,B_2)$ of~$B$ such that
$|B_1| = |A_1| + 1$ and $e$ is the only edge joining a vertex in~$A_1$
to a vertex in~$B_2$. Since $a$ and $b$ are noncubic vertices, we infer that
each of the sets $B_1$ and $A_2$ has at least three vertices. Consequently, $|B_1|=|A_2|=3$,
and each of the induced subgraphs
$G[A_1 \cup B_1]$ and $G[A_2 \cup B_2]$ is isomorphic to $K_{2,3}$.
Since $|A_1|=|B_2|=2$, by Proposition~\ref{prp:brace-characterization},
the induced subgraph~$G[A_2 \cup B_1]$ has a perfect matching.
All of these facts imply that the brace~$Q_{10}^+$ is a subgraph of~$G$,
whence $G$ is isomorphic to~$Q_{10}^+$.
\end{proof}

\begin{Def}
\label{Def:stable-extension}
{\sc [Stable-Extension]}
Let $S$ denote a stable set of a connected bipartite graph~$J[A,B]$
that meets each color class in precisely two vertices.
We let $S:=\{a_1,a_2,b_1,b_2\}$
where $a_1,a_2 \in A$ and $b_1,b_2 \in B$.
The graph~$G$ obtained from $J$ ---
by adding two new vertices $a_0$~and~$b_0$, and five new edges $a_0b_1, a_0b_2,
b_0a_1, b_0a_2$ and $a_0b_0$ --- is called the stable-extension of~$J$ with respect to~$S$ ---
or simply the $S$-extension of~$J$.
We refer to $a_0$~and~$b_0$ as the {\it extension vertices} of~$G$.
\end{Def}

\smallskip
The graph~$Q_{10}$, shown in Figure~\ref{fig:Q_10_stable},
has a unique stable set $S$ that meets each color
class in precisely two vertices. We let $Q_{12}$, shown in Figure~\ref{fig:Q_12},
denote the $S$-extension of~$Q_{10}$.
Using Propositions~\ref{prp:brace-characterization}~and~\ref{prp:brace-noncubic-stable-implies-minimal},
one may verify that $Q_{12}$ is a minimal brace.

\begin{figure}[!htb]
\centering
\subfigure[$Q_{10}$]{
\begin{tikzpicture}

\draw (-0.75,2) -- (-1.5,1) -- (0.75,2) -- (0,1) -- (-0.75,2) -- (1.5,1) -- (0.75,2);

\draw (-0.75,-1) -- (-1.5,0) -- (0.75,-1) -- (0,0) -- (-0.75,-1) -- (1.5,0) -- (0.75,-1);

\draw (-1.5,1) -- (-1.5,0);
\draw (0,1) -- (0,0);
\draw (1.5,1) -- (1.5,0);

%\draw[ultra thick] (-2,0.5) -- (2,0.5);
%\draw (-2.5,0.5) node[nodelabel]{$\partial(X)$};

\draw (-0.75,2)node{}node[above,nodelabel]{$a_1$};
\draw (0.75,2)node{}node[above,nodelabel]{$a_2$};
\draw (-1.5,0)node{};
\draw (0,0)node{};
\draw (1.5,0)node{};

\draw (-0.75,-1)node[fill=black]{}node[below,nodelabel]{$b_1$};
\draw (0.75,-1)node[fill=black]{}node[below,nodelabel]{$b_2$};
\draw (-1.5,1)node[fill=black]{};
\draw (0,1)node[fill=black]{};
\draw (1.5,1)node[fill=black]{};

\end{tikzpicture}
\label{fig:Q_10_stable}
}
\hspace*{0.5in}
\subfigure[$Q_{12}$]{
\begin{tikzpicture}

\draw (0,2) -- (-1.5,1);
\draw (0,2) -- (0,1);
\draw (0,2) -- (1.5,1);
\draw (0,2) -- (3,1);
\draw (1.5,2) -- (-1.5,1);
\draw (1.5,2) -- (0,1);
\draw (1.5,2) -- (1.5,1);
\draw (1.5,2) -- (3,1);

\draw (0,-1) -- (-1.5,0);
\draw (0,-1) -- (0,0);
\draw (0,-1) -- (1.5,0);
\draw (0,-1) -- (3,0);
\draw (1.5,-1) -- (-1.5,0);
\draw (1.5,-1) -- (0,0);
\draw (1.5,-1) -- (1.5,0);
\draw (1.5,-1) -- (3,0);

\draw (-1.5,1) -- (-1.5,0);
\draw (0,1) -- (0,0);
\draw (1.5,1) -- (1.5,0);
\draw (3,1) -- (3,0);

\draw (0,2)node{}node[above,nodelabel]{$a_1$};
\draw (1.5,2)node{}node[above,nodelabel]{$a_2$};
\draw (-1.5,0)node{};
\draw (0,0)node{};
\draw (1.5,0)node{};
\draw (3,0)node{};

\draw (0,-1)node[fill=black]{}node[below,nodelabel]{$b_1$};
\draw (1.5,-1)node[fill=black]{}node[below,nodelabel]{$b_2$};
\draw (-1.5,1)node[fill=black]{};
\draw (0,1)node[fill=black]{};
\draw (1.5,1)node[fill=black]{};
\draw (3,1)node[fill=black]{};

\draw (3.4,0)node[nodelabel]{$a_0$};
\draw (3.4,1)node[nodelabel]{$b_0$};

\end{tikzpicture}
\label{fig:Q_12}
}
\caption{The minimal brace $Q_{12}$ is the $S$-extension of $Q_{10}$ --- where $S:=\{a_1,a_2,b_1,b_2\}$}
\label{}
\end{figure}
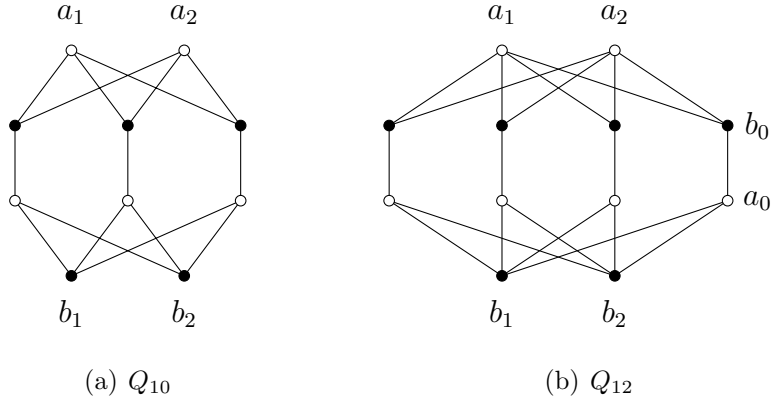

%\newpage
\begin{prp}
\label{prp:minimal-braces-order-twelve-size-twenty}
The only minimal braces of order~$12$, and size at least~$20$, are $B_{12}$~and~$Q_{12}$.
\end{prp}
\begin{proof}
Let $G[A,B]$ denote a minimal brace of order~$12$ and size at least~$20$,
and let $T$ denote the set of noncubic vertices.
Clearly, each of the sets $T \cap A$ and $T \cap B$ is nonempty.
We let $E_T$ denote the set of edges that have both ends in~$T$.

\smallskip
First suppose that $T$ is a stable set (i.e., $E_T = \emptyset$), whence $G$ has maximum degree at most five.
Observe that if there exists a vertex of degree five then $|T \cap A| = |T \cap B| = 1$;
whence $G-T$ is a connected $2$-regular bipartite graph; consequently,
$G-T$ is isomorphic to~$C_{10}$ and $G$ is isomorphic to~$B_{12}$.
Otherwise, $|T \cap A|=|T \cap B|=2$, and each member of~$T$ has degree precisely four.
In this case, observe that $G-T$
is a $1$-regular (bipartite) graph; thus, $G-T$ has four components (each
isomorphic to $K_2$), and $G$ is isomorphic to $Q_{12}$.

\smallskip
Now suppose that $T$ is not a stable set.
Our goal is to arrive at a contradiction; however, it requires some tedious arguments.

\smallskip
Let $e:=ab$ denote \underline{any} member of~$E_T$.
Since $e$ is not superfluous,
by Corollary~\ref{cor:non-superfluous-edges-in-braces},
there exist partitions $(A_1,A_2)$ of~$A$ and $(B_1,B_2)$ of~$B$ such that
$|B_1| = |A_1| + 1$ and $e$ is the only edge joining a vertex in~$A_1$
to a vertex in~$B_2$.
Since each end of~$e$ is noncubic, and since $n_G = 6$,
one of the two sets $B_1$~and~$A_2$ has cardinality three and the other one
has cardinality four. Adjust notation so that $|B_1|=3$ and $|A_2|=4$.
See Figure~\ref{fig-1:minimal-braces-order-twelve-size-twenty}.
Let $a_1$ denote the unique member of~$A_1-a$.
Observe that $N_G(a) = B_1 \cup \{b\}$ and $N_G(a_1) = B_1$.

\smallskip
In particular, we have proved the following.
\begin{sta}
\label{sta-1:minimal-braces-order-twelve-size-twenty}
Each edge~$f \in E_T$ has an end whose degree is precisely four, say~$v$,
such that there exists a cubic vertex~$v'$ that satisfies $N_G(v') = N_{G-f}(v)$. \qed
\end{sta}

The degree of vertex~$b$ is either four or five; we will prove that it must be four.
Suppose to the contrary that $b$ has degree five; whence $N_G(b) := A_2 \cup \{a\}$.
By a simple counting argument, there exists a noncubic vertex, say~$a_2$, in~$A_2$.
Since $f_2:=a_2b$ is a member of~$E_T$,
using statement~\ref{sta-1:minimal-braces-order-twelve-size-twenty}, we infer
that $a_2$ has degree precisely four, and there exists a cubic vertex~$a'$ that satisfies
$N_G(a') = N_{G-f_2}(a_2) = N_G(a_2) - b$.
Since each member of~$A_2$ is adjacent with~$b$, it follows that $a' = a_1$.
Consequently, $N_G(a_2) = B_1 \cup \{b\}$.
This implies that $N_G(B_2 - b) \subseteq A_2 -a_2$.
This contradicts Proposition~\ref{prp:brace-characterization}.
Thus $b$ has degree four.

\smallskip
In particular, we have now proved the following.
\begin{sta}
\label{sta-2:minimal-braces-order-twelve-size-twenty}
For every edge~$f \in E_T$, each end of $f$ has degree precisely four.
(Consequently, $|T \cap A|=|T \cap B| \geq 2$, and each vertex in~$T$ has degree precisely four.) \qed
\end{sta}

Now, we will prove that each neighbor of~$a$, distinct from~$b$, is cubic.
Suppose to the contrary that there exists $b_1 \in B_1$ that is noncubic.
Thus $b_1$ has degree precisely four and $f_1:=ab_1$ is a member of~$E_T$.
Now we invoke statement~\ref{sta-1:minimal-braces-order-twelve-size-twenty}.
Either there exists a cubic vertex~$a'$ that satisfies $N_G(a')=N_{G-f_1}(a) = (B_1-b_1) \cup \{b\}$,
or otherwise there exists a cubic vertex~$b'$ that satisfies $N_G(b') = N_{G-f_1}(b_1) = N_G(b_1)-a$.
In the latter case, note that $b' \in B_2$ (since each vertex in~$B_1$ is adjacent with~$a$);
however, this implies that we have an edge joining $b' \in B_2$ and $a_1 \in A_1$; contradiction.
In the former case, note that $a' \in A_2$; whence $N_G(B_2-b) \subseteq A_2 - a'$,
and this contradicts Proposition~\ref{prp:brace-characterization}.
Thus each neighbor of~$a$, distinct from~$b$, is cubic;
see Figure~\ref{fig-2:minimal-braces-order-twelve-size-twenty}.

\begin{figure}[!htb]
\centering
\subfigure[$e$ is the only edge joining $A_1$~and~$B_2$]
{
\begin{tikzpicture}[scale=0.9]

\draw (0,2) -- (0,0);
\draw (0,2) -- (1,0);
\draw (0,2) -- (2,0);

\draw (1,2) -- (0,0);
\draw (1,2) -- (1,0);
\draw (1,2) -- (2,0);

\draw (1,2) -- (4,0);
\draw (0,2.6)node[nodelabel]{$a_1$};
\draw (1,2.6)node[nodelabel]{$a$};
\draw (4,-0.6)node[nodelabel]{$b$};
\draw (1.9,1.7)node[nodelabel]{$e$};

\draw (-0.6,2)node[nodelabel]{$A_1$};
\draw (-0.6,0)node[nodelabel]{$B_1$};
\draw (6.6,0)node[nodelabel]{$B_2$};
\draw (6.6,2)node[nodelabel]{$A_2$};

\draw (-0.3,-0.3) -- (-0.3,0.3) -- (2.3,0.3) -- (2.3,-0.3) -- (-0.3,-0.3);

\draw (-0.3,1.7) -- (-0.3,2.3) -- (1.3,2.3) -- (1.3,1.7) -- (-0.3,1.7);

\draw (3.7,-0.3) -- (3.7,0.3) -- (6.3,0.3) -- (6.3,-0.3) -- (3.7,-0.3);

\draw (2.7,1.7) -- (2.7,2.3) -- (6.3,2.3) -- (6.3,1.7) -- (2.7,1.7);

\draw (4,0)node[fill=black]{};
\draw (5,0)node[fill=black]{};
\draw (6,0)node[fill=black]{};
\draw (3,2)node{};
\draw (4,2)node{};
\draw (5,2)node{};
\draw (6,2)node{};

\draw (0,0)node[fill=black]{};
\draw (1,0)node[fill=black]{};
\draw (2,0)node[fill=black]{};
\draw (0,2)node{};
\draw (1,2)node{};

\end{tikzpicture}
\label{fig-1:minimal-braces-order-twelve-size-twenty}
}
\subfigure[The number next to a vertex indicates its degree]
{
\begin{tikzpicture}[scale=0.9]

\draw (0,2) -- (0,0);
\draw (0,2) -- (1,0);
\draw (0,2) -- (2,0);

\draw (1,2) -- (0,0);
\draw (1,2) -- (1,0);
\draw (1,2) -- (2,0);

\draw (1,2) -- (4,0);
\draw (0,2.6)node[nodelabel]{$a_1$};
\draw (1,2.6)node[nodelabel]{$a$};
\draw (4,-0.6)node[nodelabel]{$b$};
\draw (1.9,1.7)node[nodelabel]{$e$};

\draw (-0.9,2)node[nodelabel]{$A_1$};
\draw (-0.9,0)node[nodelabel]{$B_1$};
\draw (6.6,0)node[nodelabel]{$B_2$};
\draw (6.6,2)node[nodelabel]{$A_2$};

\draw (-0.6,-0.3) -- (-0.6,0.3) -- (2.3,0.3) -- (2.3,-0.3) -- (-0.6,-0.3);

\draw (-0.6,1.7) -- (-0.6,2.3) -- (1.3,2.3) -- (1.3,1.7) -- (-0.6,1.7);

\draw (3.7,-0.3) -- (3.7,0.3) -- (6.3,0.3) -- (6.3,-0.3) -- (3.7,-0.3);

\draw (2.7,1.7) -- (2.7,2.3) -- (6.3,2.3) -- (6.3,1.7) -- (2.7,1.7);

\draw (4,0)node[fill=black]{};
\draw (4.3,0)node[nodelabel]{$4$};
\draw (5,0)node[fill=black]{};
\draw (6,0)node[fill=black]{};
\draw (3,2)node{};
\draw (4,2)node{};
\draw (5,2)node{};
\draw (6,2)node{};

\draw (0,0)node[fill=black]{};
\draw (-0.3,0)node[nodelabel]{$3$};
\draw (1,0)node[fill=black]{};
\draw (0.7,0)node[nodelabel]{$3$};
\draw (2,0)node[fill=black]{};
\draw (1.7,0)node[nodelabel]{$3$};
\draw (0,2)node{};
\draw (-0.3,2)node[nodelabel]{$3$};
\draw (1,2)node{};
\draw (0.7,2)node[nodelabel]{$4$};

\end{tikzpicture}
\label{fig-2:minimal-braces-order-twelve-size-twenty}
}
\caption{Illustration for the proof of
Proposition~\ref{prp:minimal-braces-order-twelve-size-twenty}}
\label{fig:minimal-braces-order-twelve-size-twenty}
\end{figure}
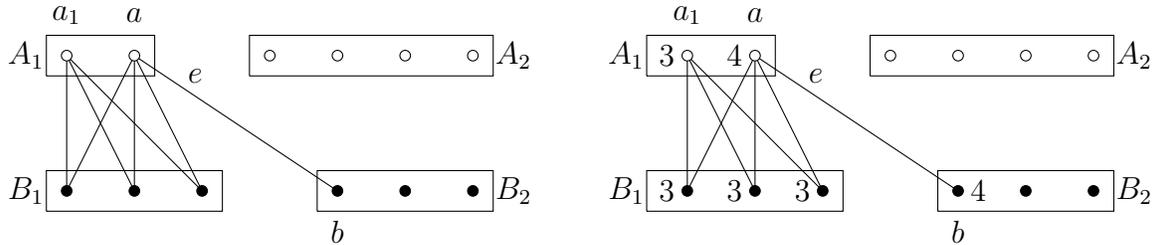

\smallskip
In particular, we have established the following.
\begin{sta}
\label{sta-3:minimal-braces-order-twelve-size-twenty}
For every edge $f \in E_T$, each end of~$f$ has degree precisely four (in~$G$) and at least one of them
has degree precisely one in the induced subgraph~$G[T]$. \qed
\end{sta}

We let $w \in A$ and $u \in B$ denote two vertices of $T$ that are distinct from $a$~and~$b$.
Note that $w \in A_2$ and $u \in B_2$.
By Proposition~\ref{prp:brace-characterization}, the graph~$G-a-a_1-b-u$ has a perfect matching;
whence the three edges joining $B_1$ and $A_2$ constitute a matching.
This implies that $w$ has precisely one neighbor in~$B_1$.
Consequently, $wb, wu \in E(G)$.
Observe that the induced subgraph $G[T]$ contains a path of length three: $(a,b,w,u)$.
The edge~$bw \in E_T$ contradicts statement~\ref{sta-3:minimal-braces-order-twelve-size-twenty}.

\smallskip
This completes the proof of Proposition~\ref{prp:minimal-braces-order-twelve-size-twenty}.
\end{proof}

\section{An induction tool for minimal braces}
\label{sec:induction-tool-minimal-braces}

In subsection~\ref{sec:McCuaig}, we
presented a `reduction version' of
McCuaig's Theorem
(\ref{thm:strictly-thin-edge-in-simple-braces})
using the notion of a strictly thin edge;
this viewpoint and the associated
terminology is due to Carvalho, Lucchesi and Murty~\cite{clm08,clm15}
and it is convenient for stating results concisely.
In the following subsection, we shall present a `generation version' of
McCuaig's Theorem (\ref{thm:generating-simple-braces}).
In order to do so, we first need to define some `expansion operations';
these will also be useful in deducing our
Main Theorem~(\ref{thm:narrow-minimality-preserving-pair-in-minimal-braces}).

\subsection{Expansion operations}
\label{sec:expansion-operations}

\begin{comment}
All graphs considered in this section are simple, bipartite and matching covered of
order six or more;
furthermore, each color class has at most one vertex of degree two.
For such a graph, say~$G$, we refer to its vertices of degree two (if any) as {\it inner vertices},
and each vertex that is a neighbor of an inner vertex is referred to as an {\it outer vertex}.
\end{comment}

For a simple bipartite connected graph~$H[A,B]$, and nonadjacent vertices $a \in A$ and $b \in B$,
$H+ab$ denotes the graph obtained from $H$ by adding the edge $ab$.
Note that if $H$ is a brace then, by Corollary~\ref{cor:adding-an-edge},
$H+ab$ is a (simple) brace and $ab$ is a strictly thin edge of index zero;
in this case, we say that $H+ab$ is obtained from $H$ by an {\it expansion of index zero}.
We shall now define two more expansion operations (on simple braces) ---
each of which may be viewed
as the reverse of removing a strictly thin edge (of index one or two)
and then taking the retract.
To do so, we first need the notion of `bi-splitting' a noncubic vertex.

\smallskip
Let $b$ denote a noncubic vertex of a simple bipartite \mcg~$H[A',B']$.
Adjust notation so that $b \in B'$.
Suppose that a (bipartite) graph $G$ is obtained from $H$ by replacing the vertex~$b$ by two new
vertices $b_1$~and~$b_2$, distributing the edges in $H$ incident with~$b$ between $b_1$~and~$b_2$
such that each gets at least two edges, and then adding a new vertex~$a_0$ and two new
edges: $a_0b_1, a_0b_2$.
We say that $G$ is obtained from $H$ by {\it bi-splitting $b$ into $b_1a_0b_2$},
and we denote this as $G:=H\{b \rightarrow b_1a_0b_2\}$.
It is easily verified that $G$ is also \mc.
Observe that $H$ can be recovered from $G$ by bicontracting the vertex~$a_0$
and denoting the contraction vertex by~$b$.

\smallskip
We are now ready to define the aforementioned expansion operations.
(See Figures \ref{fig:index-one} and \ref{fig:index-two}.)

\begin{figure}[!htb]
\centering
\subfigure[brace~$H:=\widehat{G-e}$]
{
\begin{tikzpicture}

\draw (1.5,2.75) -- (0.5,1.5);
\draw (1.5,2.75) -- (1,1.5);
\draw (1.5,2.75) -- (2,1.5);
\draw (1.5,2.75) -- (2.5,1.5);

\draw (1.5,2.75)node{}node[nodelabel,above]{$a$};
\draw (1.5,0.75)node{};%node[nodelabel,below]{$w$};
\draw (1.5,0.45)node[nodelabel]{$w$};

\draw (0.5,1.5)node[fill=black]{};
\draw (1,1.5)node[fill=black]{};
\draw (2,1.5)node[fill=black]{};
\draw (2.5,1.5)node[fill=black]{};

\draw (0,0) -- (3,0) -- (3,2) -- (0,2) -- (0,0);
\end{tikzpicture}
%\label{fig:}
}
\hspace*{1in}
\subfigure[brace~$G$]
{
\begin{tikzpicture}

\draw (1.5,3.5) to [out=0,in=90] (3.5,2) to [out=270,in=0] (1.5,0.75);
\draw (3.7,2)node[nodelabel]{$e$};

\draw (1.5,3.5) -- (1,2.75);
\draw (1.5,3.5) -- (2,2.75);

\draw (1,2.75) -- (0.5,1.5);
\draw (1,2.75) -- (1,1.5);
\draw (2,2.75) -- (2,1.5);
\draw (2,2.75) -- (2.5,1.5);

\draw (1,2.75)node{};%node[nodelabel,left]{$a_1$};
\draw (0.7,2.75)node[nodelabel]{$a_1$};
\draw (2,2.75)node{};%node[nodelabel,right]{$a_2$};
\draw (2.35,2.75)node[nodelabel]{$a_2$};

\draw (1.5,3.5)node[fill=black]{};%node[nodelabel,above]{$b_0$};
\draw (1.5,3.85)node[nodelabel]{$b_0$};
\draw (1.5,0.75)node{};%node[nodelabel,below]{$w$};
\draw (1.5,0.45)node[nodelabel]{$w$};

\draw (0.5,1.5)node[fill=black]{};
\draw (1,1.5)node[fill=black]{};
\draw (2,1.5)node[fill=black]{};
\draw (2.5,1.5)node[fill=black]{};

\draw (0,0) -- (3,0) -- (3,2) -- (0,2) -- (0,0);
\end{tikzpicture}
%\label{fig:}
}
\caption{$G$ is obtained from~$H$ by an expansion of index one}
\label{fig:index-one}
\end{figure}
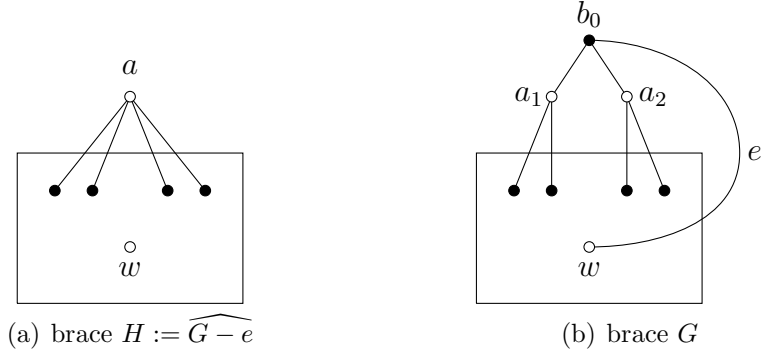

\begin{Def}
\label{Def:expansion-index-one}
{\sc [Expansion of Index One]}
Choose two vertices, say~$a$~and~$w$,
of a simple brace~$H$ that belong to the same color class
such that at least one of them, say~$a$, is noncubic.
A graph~$G$ is obtained from~$H$ by an expansion of index one
if $G:=H\{a \rightarrow a_1b_0a_2\} + b_0w$.
\end{Def}

As an example, one may construct the brace~$Q_{12}$ from the brace~$Q_{10}^+$
by means of an expansion of index zero (i.e., adding an edge) followed by an expansion of index one.
To see this, we let $Q_{10}^+ := Q_{10} + a_1b_1$ as per the labeling in Figure~\ref{fig:Q_10_stable}.
Now, observe that, as per the labeling in Figure~\ref{fig:Q_12},
$Q_{12} := (Q_{10}^+ + a_1b_2)\{a_1 \rightarrow a_1b_0a_0 \} + b_0a_2$.\footnote{In order to be consistent with the labeling used in Figures~\ref{fig:Q_10_stable}~and~\ref{fig:Q_12},
we slightly abuse notation by overloading the label $a_1$.}

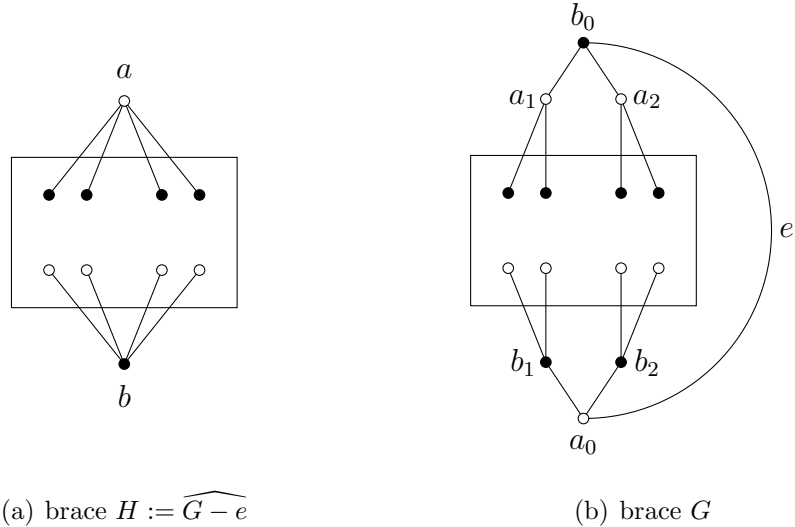
\begin{figure}[!htb]
\centering
\subfigure[brace~$H:=\widehat{G-e}$]
{
\begin{tikzpicture}

\draw (1.5,-0.75) -- (0.5,0.5);
\draw (1.5,-0.75) -- (1,0.5);
\draw (1.5,-0.75) -- (2,0.5);
\draw (1.5,-0.75) -- (2.5,0.5);

\draw (1.5,2.75) -- (0.5,1.5);
\draw (1.5,2.75) -- (1,1.5);
\draw (1.5,2.75) -- (2,1.5);
\draw (1.5,2.75) -- (2.5,1.5);

\draw (1.5,2.75)node{}node[nodelabel,above]{$a$};

\draw (1.5,-0.75)node[fill=black]{}node[nodelabel,below]{$b$};

\draw (0.5,1.5)node[fill=black]{};
\draw (1,1.5)node[fill=black]{};
\draw (2,1.5)node[fill=black]{};
\draw (2.5,1.5)node[fill=black]{};

\draw (0.5,0.5)node{};
\draw (1,0.5)node{};
\draw (2,0.5)node{};
\draw (2.5,0.5)node{};

\draw (0,0) -- (3,0) -- (3,2) -- (0,2) -- (0,0);

\draw (1.5,-2.05)node[nodelabel]{};
\end{tikzpicture}
%\label{fig:}
}
\hspace*{1in}
\subfigure[brace~$G$]
{
\begin{tikzpicture}

\draw (1.5,3.5) to [out=0,in=90] (4,1) to [out=270,in=0] (1.5,-1.5);
\draw (4.2,1)node[nodelabel]{$e$};

\draw (1.5,-1.5) -- (1,-0.75);
\draw (1.5,-1.5) -- (2,-0.75);

\draw (1.5,3.5) -- (1,2.75);
\draw (1.5,3.5) -- (2,2.75);

\draw (1,2.75) -- (0.5,1.5);
\draw (1,2.75) -- (1,1.5);
\draw (2,2.75) -- (2,1.5);
\draw (2,2.75) -- (2.5,1.5);

\draw (1,-0.75) -- (0.5,0.5);
\draw (1,-0.75) -- (1,0.5);
\draw (2,-0.75) -- (2,0.5);
\draw (2,-0.75) -- (2.5,0.5);

\draw (1,2.75)node{};
\draw (0.7,2.75)node[nodelabel]{$a_1$};
\draw (2,2.75)node{};
\draw (2.35,2.75)node[nodelabel]{$a_2$};

\draw (1,-0.75)node[fill=black]{};
\draw (0.7,-0.75)node[nodelabel]{$b_1$};
\draw (2,-0.75)node[fill=black]{};
\draw (2.35,-0.75)node[nodelabel]{$b_2$};

\draw (1.5,3.5)node[fill=black]{};
\draw (1.5,3.85)node[nodelabel]{$b_0$};

\draw (1.5,-1.5)node{};
\draw (1.5,-1.85)node[nodelabel]{$a_0$};

\draw (0.5,1.5)node[fill=black]{};
\draw (1,1.5)node[fill=black]{};
\draw (2,1.5)node[fill=black]{};
\draw (2.5,1.5)node[fill=black]{};

\draw (0.5,0.5)node{};
\draw (1,0.5)node{};
\draw (2,0.5)node{};
\draw (2.5,0.5)node{};

\draw (0,0) -- (3,0) -- (3,2) -- (0,2) -- (0,0);
\end{tikzpicture}
%\label{fig:}
}
\caption{$G$ is obtained from~$H$ by an expansion of index two}
\label{fig:index-two}
\end{figure}

\begin{Def}
\label{Def:expansion-index-two}
{\sc [Expansion of Index Two]}
Choose two noncubic vertices, say~$a$~and~$b$,
of a simple brace~$H$ that lie in distinct color classes.
A graph~$G$ is obtained from~$H$ by an expansion of index two
if $G:=H\{a \rightarrow a_1b_0a_2\}\{b \rightarrow b_1a_0b_2\} + a_0b_0$.
\end{Def}

%\smallskip
For convenience, we say that a graph~$G$ is obtained from the simple brace~$H$ by an {\it expansion operation}
if $G$ is obtained from~$H$ by an expansion of index zero, one or two.
McCuaig~\cite{mccu01} proved the following.

\begin{prp}
\label{prp:expansion-operations-preserve-brace-property}
Any graph~$G$, that is obtained from a simple brace~$H$ by an expansion operation,
is also a simple brace.
\end{prp}

Let $H[A,B]$ denote a simple brace. Observe the following.
If $G:=H\{a \rightarrow a_1b_0a_2\} + b_0w$ for some $a,w \in A$,
then $e:=b_0w$ is a strictly thin edge of index one.
On the other hand, if $G:=H\{a \rightarrow a_1b_0a_2\}\{b \rightarrow b_1a_0b_2\} + a_0b_0$
for some $a \in A$ and $b \in B$,
then $e:=a_0b_0$ is a strictly thin edge of index two.
In either case, $H$ is isomorphic to $\widehat{G-e}$.
We now state the `generation version' of McCuaig's Theorem (that is equivalent to
Theorem~\ref{thm:strictly-thin-edge-in-simple-braces}.)

\begin{thm}
\label{thm:generating-simple-braces}
Every simple brace $G \notin \mathcal{G}$ may be obtained from a
(smaller) simple brace by an expansion operation.
\end{thm}

In the following two subsections, we state and prove several lemmas;
these will culminate in the proof of the
Main Theorem~(\ref{thm:narrow-minimality-preserving-pair-in-minimal-braces})
that appears in subsection~\ref{sec:main-theorem}.
However, we need some more terminology in order to state these lemmas.

\smallskip
Consider the graph~$G-e$ where $e$ is a strictly thin edge of a simple brace~$G$.
In~$G-e$, a vertex of degree two is referred to as an {\it inner vertex},
and each of its neighbors is referred to as an {\it outer vertex}.
Note that $G-e$ has precisely ${\rm index}(e)$ inner vertices
and $2 \cdot {\rm index}(e)$ outer vertices; see Figures~\ref{fig:index-one}~and~\ref{fig:index-two}.

%\newpage

%\newpage
\subsection{Index one}
\label{sec:index-one}

Throughout this subsection, we assume that $G[A,B] \notin \mathcal{G}$
is a minimal brace, and that $e:=b_0w$ is a strictly thin edge of index one
where $b_0 \in B$ is the cubic end of~$e$.
We also assume that the simple brace $H:=\widehat{G-e}$ is not minimal;
whence $H$ has superfluous edge(s).
Thus $n_H \geq 4$ and $n_G \geq 5$.
We let $a_1$ and $a_2$ denote
the outer vertices of~$G-e$. See Figure~\ref{fig:index-one}.

\begin{lem}
\label{lem-1:index-one}
Let $f$ denote any superfluous edge of~$H$. Then $G-e$ has a cubic outer vertex
that is an end of~$f$, whereas the other end of~$f$ is noncubic.
\end{lem}
\begin{proof}
Recall definition~\ref{Def:expansion-index-one}, and observe the following.
If each of $a_1$ and $a_2$ has degree at least three in~$G-f$,
then $G-f$ may be obtained from the brace~$H-f$ by an expansion of index one;
whence $G-f$ is a brace (by Proposition~\ref{prp:expansion-operations-preserve-brace-property});
contradiction.
It follows that one of $a_1$ and $a_2$ is a cubic end of~$f$.
Clearly, the other end of~$f$ is noncubic.
\end{proof}

\begin{lem}
\label{lem-2:index-one}
Let $f_1$ and $f_2$ denote two edges of~$H$ such that $H-f_1-f_2$
is a brace.
Then $G-e$ has a cubic outer vertex that is a common end of $f_1$~and~$f_2$.
\end{lem}
\begin{proof}
It follows from the hypothesis that each of $f_1$ and $f_2$ is a superfluous edge of~$H$.
By Lemma~\ref{lem-1:index-one}, for $i \in \{1,2\}$, the edge $f_i$ has precisely one cubic end
(in~$G$) that
lies in $\{a_1,a_2\}$.
If the cubic end of $f_1$ is the same as the cubic end of~$f_2$,
then there is nothing to prove.

\smallskip
Now suppose that $f_i \in \partial(a_i)$ for $i \in \{1,2\}$.
Consequently, $a_1$ and $a_2$ are both cubic (in~$G$), whence $a$ has degree
precisely four in~$H$; consequently, $a$ is a vertex of degree two in~$H-f_1-f_2$;
contradiction.
\end{proof}

Note that, if $(e,F)$ is a minimality-preserving pair of~$G$ then any two distinct members
of~$F$, say~$f_1$ and $f_2$, satisfy the hypothesis of Lemma~\ref{lem-2:index-one}.
We thus have the following immediate consequence.

\begin{cor}
\label{cor:index-one}
If $(e,F)$ is a minimality-preserving pair of~$G$, then
$G-e$ has a cubic outer vertex~$v$ such that $F \subset \partial(v)$.
Consequently, $|F| \in \{1,2\}$.
\qed
\end{cor}

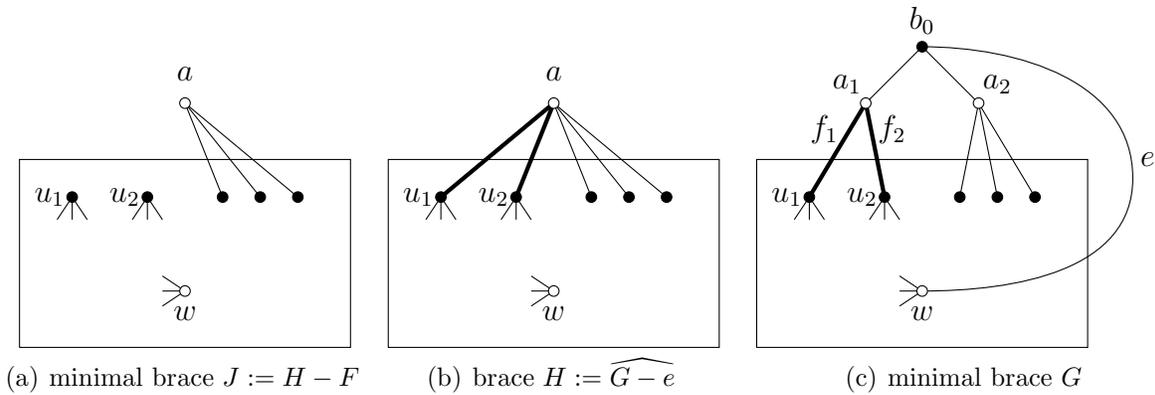
\begin{figure}[!htb]
\centering
\subfigure[minimal brace $J:=H-F$]
{
\begin{tikzpicture}[scale=1]

\draw (2,0.25) -- (1.7,0.05);
\draw (2,0.25) -- (1.7,0.25);
\draw (2,0.25) -- (1.7,0.45);

\draw (0.5,1.5) -- (0.3,1.2);
\draw (0.5,1.5) -- (0.5,1.2);
\draw (0.5,1.5) -- (0.7,1.2);

\draw (1.5,1.5) -- (1.3,1.2);
\draw (1.5,1.5) -- (1.5,1.2);
\draw (1.5,1.5) -- (1.7,1.2);

\draw (2,2.75) -- (2.5,1.5);
\draw (2,2.75) -- (3,1.5);
\draw (2,2.75) -- (3.5,1.5);

\draw (2,2.75)node{}node[nodelabel,above]{$a$};

\draw (0.5,1.5)node[fill=black]{};
\draw (1.5,1.5)node[fill=black]{};
\draw (2.5,1.5)node[fill=black]{};
\draw (3,1.5)node[fill=black]{};
\draw (3.5,1.5)node[fill=black]{};

\draw (0.2,1.5)node[nodelabel]{$u_1$};
\draw (1.2,1.5)node[nodelabel]{$u_2$};

\draw (2,0.25)node{};
\draw (2,-0.05)node[nodelabel]{$w$};
\draw (-0.2,-0.5) -- (4.2,-0.5) -- (4.2,2) -- (-0.2,2) -- (-0.2,-0.5);
\end{tikzpicture}
%\label{fig:}
}
\subfigure[brace~$H:=\widehat{G-e}$]
{
\begin{tikzpicture}[scale=1]

\draw (2,0.25) -- (1.7,0.05);
\draw (2,0.25) -- (1.7,0.25);
\draw (2,0.25) -- (1.7,0.45);

\draw (0.5,1.5) -- (0.3,1.2);
\draw (0.5,1.5) -- (0.5,1.2);
\draw (0.5,1.5) -- (0.7,1.2);

\draw (1.5,1.5) -- (1.3,1.2);
\draw (1.5,1.5) -- (1.5,1.2);
\draw (1.5,1.5) -- (1.7,1.2);

\draw[ultra thick] (2,2.75) -- (0.5,1.5);
\draw[ultra thick] (2,2.75) -- (1.5,1.5);
\draw (2,2.75) -- (2.5,1.5);
\draw (2,2.75) -- (3,1.5);
\draw (2,2.75) -- (3.5,1.5);

\draw (2,2.75)node{}node[nodelabel,above]{$a$};

\draw (0.5,1.5)node[fill=black]{};
\draw (1.5,1.5)node[fill=black]{};
\draw (2.5,1.5)node[fill=black]{};
\draw (3,1.5)node[fill=black]{};
\draw (3.5,1.5)node[fill=black]{};

\draw (0.2,1.5)node[nodelabel]{$u_1$};
\draw (1.2,1.5)node[nodelabel]{$u_2$};

\draw (2,0.25)node{};
\draw (2,-0.05)node[nodelabel]{$w$};
\draw (-0.2,-0.5) -- (4.2,-0.5) -- (4.2,2) -- (-0.2,2) -- (-0.2,-0.5);
\end{tikzpicture}
%\label{fig:}
}
%\hspace*{1in}
\subfigure[minimal brace~$G$]
{
\begin{tikzpicture}[scale=1]

\draw (2,0.25) -- (1.7,0.05);
\draw (2,0.25) -- (1.7,0.25);
\draw (2,0.25) -- (1.7,0.45);

\draw (0.5,1.5) -- (0.3,1.2);
\draw (0.5,1.5) -- (0.5,1.2);
\draw (0.5,1.5) -- (0.7,1.2);

\draw (1.5,1.5) -- (1.3,1.2);
\draw (1.5,1.5) -- (1.5,1.2);
\draw (1.5,1.5) -- (1.7,1.2);

\draw (2,3.5) to [out=0,in=90] (4.8,1.75) to [out=270,in=0] (2,0.25);
\draw (5,2)node[nodelabel]{$e$};

\draw (2,3.5) -- (1.25,2.75);
\draw (2,3.5) -- (2.75,2.75);

\draw (0.7,2.4)node[nodelabel]{$f_1$};
\draw (1.6,2.4)node[nodelabel]{$f_2$};

\draw[ultra thick] (1.25,2.75) -- (0.5,1.5);
\draw[ultra thick] (1.25,2.75) -- (1.5,1.5);
\draw (2.75,2.75) -- (2.5,1.5);
\draw (2.75,2.75) -- (3,1.5);
\draw (2.75,2.75) -- (3.5,1.5);

\draw (1.25,2.75)node{};
\draw (1,3)node[nodelabel]{$a_1$};
\draw (2.75,2.75)node{};
\draw (3,3)node[nodelabel]{$a_2$};

\draw (2,3.5)node[fill=black]{};%node[nodelabel,above]{$b_0$};
\draw (2,3.85)node[nodelabel]{$b_0$};

\draw (0.5,1.5)node[fill=black]{};
\draw (1.5,1.5)node[fill=black]{};
\draw (2.5,1.5)node[fill=black]{};
\draw (3,1.5)node[fill=black]{};
\draw (3.5,1.5)node[fill=black]{};

\draw (0.2,1.5)node[nodelabel]{$u_1$};
\draw (1.2,1.5)node[nodelabel]{$u_2$};

\draw (2,0.25)node{};
\draw (2,-0.05)node[nodelabel]{$w$};
\draw (-0.2,-0.5) -- (4.2,-0.5) -- (4.2,2) -- (-0.2,2) -- (-0.2,-0.5);

\end{tikzpicture}
%\label{fig:}
}
\caption{$(e,F)$ is a minimality-preserving pair where ${\rm index}(e)=1$ and $F=\{f_1,f_2\}$}
\label{fig:index-one-F-2}
\end{figure}

\begin{lem}
\label{lem-3:index-one}
Suppose that $(e,F)$ is a minimality-preserving pair of~$G$ such that $|F|=2$.
Then, for each $f \in F$, the noncubic end of~$f$ is not adjacent with the noncubic
end of~$e$. Moreover, $G$ is isomorphic to a
stable-extension of the minimal brace~$J:=\widehat{G-e}-F$.
\end{lem}
\begin{proof}
Let $F:=\{f_1,f_2\}$. We invoke Lemma~\ref{lem-2:index-one}, and adjust notation so
that $a_1$ is cubic and $f_1,f_2 \in \partial(a_1)$. We let $u_1$ and $u_2$ denote the
noncubic ends of $f_1$~and~$f_2$, respectively. We let $J:=H-F$. Thus $J$ is a minimal brace.
See Figure~\ref{fig:index-one-F-2}.

\smallskip
Our goal is to prove that $wu_1, wu_2 \notin E(G)$.
By symmetry, it suffices to prove that $wu_2 \notin E(G)$.
Suppose to the contrary that $wu_2 \in E(G)$.
See Figure~\ref{fig:index-one-F-2-adjacency}.

\begin{figure}[!htb]
\centering
\subfigure[minimal brace~$G$]
{
\begin{tikzpicture}[scale=1]

\draw (2,0.25) to [out=90,in=0] (1.5,1.5);

\draw (2,0.25) -- (1.7,0.05);
\draw (2,0.25) -- (1.7,0.25);

\draw (0.5,1.5) -- (0.3,1.2);
\draw (0.5,1.5) -- (0.5,1.2);
\draw (0.5,1.5) -- (0.7,1.2);

\draw (1.5,1.5) -- (1.3,1.2);
\draw (1.5,1.5) -- (1.5,1.2);

\draw (2,3.5) to [out=0,in=90] (4.8,1.75) to [out=270,in=0] (2,0.25);
\draw (5,2)node[nodelabel]{$e$};

\draw (2,3.5) -- (1.25,2.75);
\draw (2,3.5) -- (2.75,2.75);

\draw (0.7,2.4)node[nodelabel]{$f_1$};
\draw (1.6,2.4)node[nodelabel]{$f_2$};

\draw[ultra thick] (1.25,2.75) -- (0.5,1.5);
\draw[ultra thick] (1.25,2.75) -- (1.5,1.5);
\draw (2.75,2.75) -- (2.5,1.5);
\draw (2.75,2.75) -- (3,1.5);
\draw (2.75,2.75) -- (3.5,1.5);

\draw (1.25,2.75)node{};
\draw (1,3)node[nodelabel]{$a_1$};
\draw (2.75,2.75)node{};
\draw (3,3)node[nodelabel]{$a_2$};

\draw (2,3.5)node[fill=black]{};%node[nodelabel,above]{$b_0$};
\draw (2,3.85)node[nodelabel]{$b_0$};

\draw (0.5,1.5)node[fill=black]{};
\draw (1.5,1.5)node[fill=black]{};
\draw (2.5,1.5)node[fill=black]{};
\draw (3,1.5)node[fill=black]{};
\draw (3.5,1.5)node[fill=black]{};

\draw (0.2,1.5)node[nodelabel]{$u_1$};
\draw (1.2,1.5)node[nodelabel]{$u_2$};

\draw (2,0.25)node{};
\draw (2,-0.05)node[nodelabel]{$w$};
\draw (-0.2,-0.5) -- (4.2,-0.5) -- (4.2,2) -- (-0.2,2) -- (-0.2,-0.5);

\end{tikzpicture}
%\label{fig:}
}
\subfigure[minimal brace $J:=H-F$]
{
\begin{tikzpicture}[scale=1]

\draw (2,0.25) to [out=90,in=0] (1.5,1.5);

\draw (2,0.25) -- (1.7,0.05);
\draw (2,0.25) -- (1.7,0.25);

\draw (0.5,1.5) -- (0.3,1.2);
\draw (0.5,1.5) -- (0.5,1.2);
\draw (0.5,1.5) -- (0.7,1.2);

\draw (1.5,1.5) -- (1.3,1.2);
\draw (1.5,1.5) -- (1.5,1.2);

\draw (2,2.75) -- (2.5,1.5);
\draw (2,2.75) -- (3,1.5);
\draw (2,2.75) -- (3.5,1.5);

\draw (2,2.75)node{}node[nodelabel,above]{$a$};

\draw (0.5,1.5)node[fill=black]{};
\draw (1.5,1.5)node[fill=black]{};
\draw (2.5,1.5)node[fill=black]{};
\draw (3,1.5)node[fill=black]{};
\draw (3.5,1.5)node[fill=black]{};

\draw (0.2,1.5)node[nodelabel]{$u_1$};
\draw (1.2,1.5)node[nodelabel]{$u_2$};

\draw (2,0.25)node{};
\draw (2,-0.05)node[nodelabel]{$w$};
\draw (-0.2,-0.5) -- (4.2,-0.5) -- (4.2,2) -- (-0.2,2) -- (-0.2,-0.5);
\end{tikzpicture}
%\label{fig:}
}
\subfigure[brace $J':=J+au_2$]
{
\begin{tikzpicture}[scale=1]

\draw (2,2.75) -- (1.5,1.5);

\draw (2,0.25) to [out=90,in=0] (1.5,1.5);

\draw (2,0.25) -- (1.7,0.05);
\draw (2,0.25) -- (1.7,0.25);

\draw (0.5,1.5) -- (0.3,1.2);
\draw (0.5,1.5) -- (0.5,1.2);
\draw (0.5,1.5) -- (0.7,1.2);

\draw (1.5,1.5) -- (1.3,1.2);
\draw (1.5,1.5) -- (1.5,1.2);

\draw (2,2.75) -- (2.5,1.5);
\draw (2,2.75) -- (3,1.5);
\draw (2,2.75) -- (3.5,1.5);

\draw (2,2.75)node{}node[nodelabel,above]{$a$};

\draw (0.5,1.5)node[fill=black]{};
\draw (1.5,1.5)node[fill=black]{};
\draw (2.5,1.5)node[fill=black]{};
\draw (3,1.5)node[fill=black]{};
\draw (3.5,1.5)node[fill=black]{};

\draw (0.2,1.5)node[nodelabel]{$u_1$};
\draw (1.2,1.5)node[nodelabel]{$u_2$};

\draw (2,0.25)node{};
\draw (2,-0.05)node[nodelabel]{$w$};
\draw (-0.2,-0.5) -- (4.2,-0.5) -- (4.2,2) -- (-0.2,2) -- (-0.2,-0.5);
\end{tikzpicture}
%\label{fig:}
}
\caption{Illustration for the proof of Lemma~\ref{lem-3:index-one}}
\label{fig:index-one-F-2-adjacency}
\end{figure}
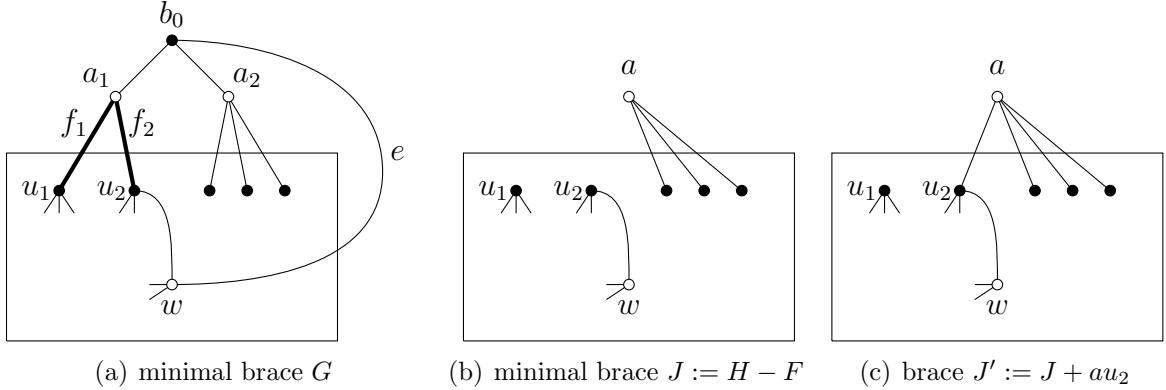

In this case, $w$ and $u_2$ are adjacent in~$J$ as well.
Observe that one may obtain the graph $G':=G-wu_2$ from the brace~$J':=J+au_2$
by an expansion of index one.
In particular, $G'=J' \{u_2 \rightarrow u_2a_1b_0 \} + a_1u_1$.
By Proposition~\ref{prp:expansion-operations-preserve-brace-property},
$G'$ is a brace --- contrary to our hypothesis that $G$ is minimal.

\smallskip
We thus conclude that $wu_1, wu_2 \notin E(G)$.
Consequently, $S:= \{a,w,u_1,u_2\}$ is a stable set of~$J$ that meets each color
class in two vertices. Observe that $G$ is isomorphic to the $S$-extension of~$J$
(with $a_1$ and $b_0$ playing the role of the extension vertices).
This completes the proof of Lemma~\ref{lem-3:index-one}.
\end{proof}

\subsection{Index two}
\label{sec:index-two}

Throughout this subsection, we assume that $G[A,B] \notin \mathcal{G}$ is a minimal brace
that is devoid of strictly thin edges of index one,
and that \mbox{$e:=a_0b_0$} is a strictly thin edge of index two.
We also assume that the simple brace \mbox{$H:=\widehat{G-e}$} is not minimal;
whence $H$ has superfluous edge(s).
Thus $n_H \geq 4$ and $n_G \geq 6$. We let $a_1,a_2 \in A$ and $b_1,b_2 \in B$
denote the outer vertices of~$G-e$. See Figure~\ref{fig:index-two}.

\smallskip
Since $e$ is a strictly thin edge, $G$ has at most one edge that has one end in~$\{a_1,a_2\}$
and the other end in~$\{b_1,b_2\}$; furthermore, $G$ has (precisely) one such edge
if and only if $ab \in E(H)$.

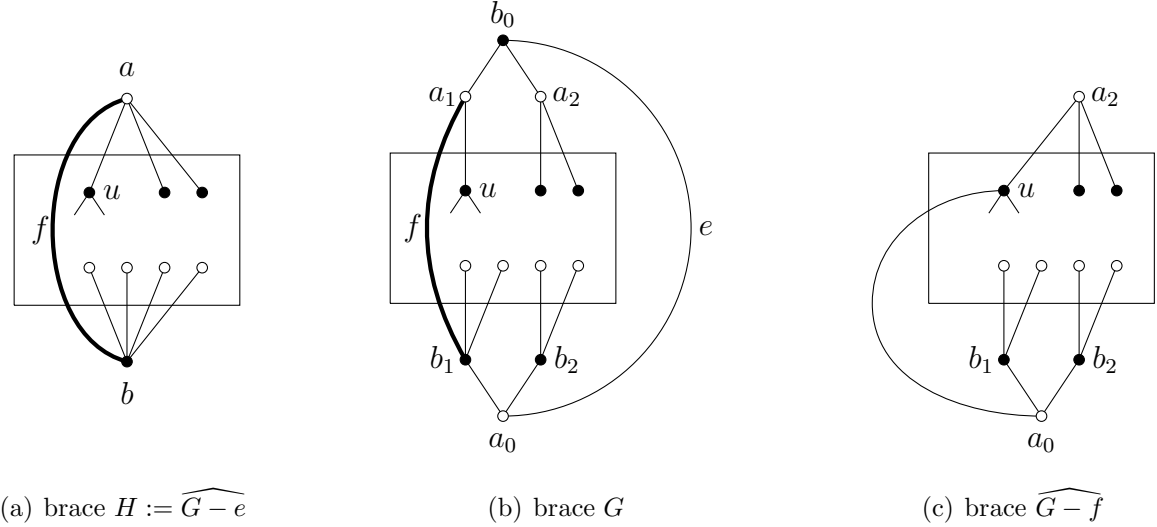
\begin{figure}[!htb]
\centering
\subfigure[brace~$H:=\widehat{G-e}$]
{
\begin{tikzpicture}

\draw (1.3,1.5)node[nodelabel]{$u$};

\draw (1,1.5) -- (0.8,1.2);
\draw (1,1.5) -- (1.2,1.2);

\draw[ultra thick] (1.5,2.75) to [out=195,in=165] (1.5,-0.75);
\draw (0.35,1)node[nodelabel]{$f$};

%\draw (1.5,-0.75) -- (0.5,0.5);
\draw (1.5,-0.75) -- (1,0.5);
\draw (1.5,-0.75) -- (1.5,0.5);
\draw (1.5,-0.75) -- (2,0.5);
\draw (1.5,-0.75) -- (2.5,0.5);

%\draw (1.5,2.75) -- (0.5,1.5);
\draw (1.5,2.75) -- (1,1.5);
\draw (1.5,2.75) -- (2,1.5);
\draw (1.5,2.75) -- (2.5,1.5);

\draw (1.5,2.75)node{}node[nodelabel,above]{$a$};

\draw (1.5,-0.75)node[fill=black]{}node[nodelabel,below]{$b$};

%\draw (0.5,1.5)node[fill=black]{};
\draw (1,1.5)node[fill=black]{};
\draw (2,1.5)node[fill=black]{};
\draw (2.5,1.5)node[fill=black]{};

%\draw (0.5,0.5)node{};
\draw (1,0.5)node{};
\draw (1.5,0.5)node{};
\draw (2,0.5)node{};
\draw (2.5,0.5)node{};

\draw (0,0) -- (3,0) -- (3,2) -- (0,2) -- (0,0);

\draw (1.5,-2.05)node[nodelabel]{};
\end{tikzpicture}
%\label{fig:}
}
\hspace*{0.5in}
\subfigure[brace~$G$]
{
\begin{tikzpicture}

\draw (1.3,1.5)node[nodelabel]{$u$};

\draw (1,1.5) -- (0.8,1.2);
\draw (1,1.5) -- (1.2,1.2);

\draw[ultra thick] (1,2.75) to [out=240,in=120] (1,-0.75);
\draw (0.3,1)node[nodelabel]{$f$};

\draw (1.5,3.5) to [out=0,in=90] (4,1) to [out=270,in=0] (1.5,-1.5);
\draw (4.2,1)node[nodelabel]{$e$};

\draw (1.5,-1.5) -- (1,-0.75);
\draw (1.5,-1.5) -- (2,-0.75);

\draw (1.5,3.5) -- (1,2.75);
\draw (1.5,3.5) -- (2,2.75);

%\draw (1,2.75) -- (0.5,1.5);
\draw (1,2.75) -- (1,1.5);
\draw (2,2.75) -- (2,1.5);
\draw (2,2.75) -- (2.5,1.5);

%\draw (1,-0.75) -- (0.5,0.5);
\draw (1,-0.75) -- (1,0.5);
\draw (1,-0.75) -- (1.5,0.5);
\draw (2,-0.75) -- (2,0.5);
\draw (2,-0.75) -- (2.5,0.5);

\draw (1,2.75)node{};
\draw (0.7,2.75)node[nodelabel]{$a_1$};
\draw (2,2.75)node{};
\draw (2.35,2.75)node[nodelabel]{$a_2$};

\draw (1,-0.75)node[fill=black]{};
\draw (0.7,-0.75)node[nodelabel]{$b_1$};
\draw (2,-0.75)node[fill=black]{};
\draw (2.35,-0.75)node[nodelabel]{$b_2$};

\draw (1.5,3.5)node[fill=black]{};
\draw (1.5,3.85)node[nodelabel]{$b_0$};

\draw (1.5,-1.5)node{};
\draw (1.5,-1.85)node[nodelabel]{$a_0$};

%\draw (0.5,1.5)node[fill=black]{};
\draw (1,1.5)node[fill=black]{};
\draw (2,1.5)node[fill=black]{};
\draw (2.5,1.5)node[fill=black]{};

%\draw (0.5,0.5)node{};
\draw (1,0.5)node{};
\draw (1.5,0.5)node{};
\draw (2,0.5)node{};
\draw (2.5,0.5)node{};

\draw (0,0) -- (3,0) -- (3,2) -- (0,2) -- (0,0);
\end{tikzpicture}
%\label{fig:}
}
\hspace*{0.5in}
\subfigure[brace~$\widehat{G-f}$]
{
\begin{tikzpicture}

\draw (1.3,1.5)node[nodelabel]{$u$};

\draw (2,2.75) -- (1,1.5);
\draw (1,1.5) to [out=180,in=90] (-0.75,0) to [out=270,in=180] (1.5,-1.5);

\draw (1,1.5) -- (0.8,1.2);
\draw (1,1.5) -- (1.2,1.2);

%\draw (1,2.75) to [out=240,in=120] (1,-0.75);
%\draw (0.3,1)node[nodelabel]{$f$};

%\draw (1.5,3.5) to [out=0,in=90] (4,1) to [out=270,in=0] (1.5,-1.5);
%\draw (4.2,1)node[nodelabel]{$e$};

\draw (1.5,-1.5) -- (1,-0.75);
\draw (1.5,-1.5) -- (2,-0.75);

%\draw (1.5,3.5) -- (1,2.75);
%\draw (1.5,3.5) -- (2,2.75);

%\draw (1,2.75) -- (0.5,1.5);
%\draw (1,2.75) -- (1,1.5);
\draw (2,2.75) -- (2,1.5);
\draw (2,2.75) -- (2.5,1.5);

%\draw (1,-0.75) -- (0.5,0.5);
\draw (1,-0.75) -- (1,0.5);
\draw (1,-0.75) -- (1.5,0.5);
\draw (2,-0.75) -- (2,0.5);
\draw (2,-0.75) -- (2.5,0.5);

%\draw (1,2.75)node{};
%\draw (0.7,2.75)node[nodelabel]{$a_1$};
\draw (2,2.75)node{};
\draw (2.35,2.75)node[nodelabel]{$a_2$};

\draw (1,-0.75)node[fill=black]{};
\draw (0.7,-0.75)node[nodelabel]{$b_1$};
\draw (2,-0.75)node[fill=black]{};
\draw (2.35,-0.75)node[nodelabel]{$b_2$};

%\draw (1.5,3.5)node[fill=black]{};
%\draw (1.5,3.85)node[nodelabel]{$b_0$};

\draw (1.5,-1.5)node{};
\draw (1.5,-1.85)node[nodelabel]{$a_0$};

%\draw (0.5,1.5)node[fill=black]{};
\draw (1,1.5)node[fill=black]{};
\draw (2,1.5)node[fill=black]{};
\draw (2.5,1.5)node[fill=black]{};

%\draw (0.5,0.5)node{};
\draw (1,0.5)node{};
\draw (1.5,0.5)node{};
\draw (2,0.5)node{};
\draw (2.5,0.5)node{};

\draw (0,0) -- (3,0) -- (3,2) -- (0,2) -- (0,0);
\end{tikzpicture}
%\label{fig:}
}
\caption{Illustration for the proof of Lemma~\ref{lem-1:index-two}}
\label{fig-1:index-two}
\end{figure}

\begin{lem}
\label{lem-1:index-two}
Let $f$ denote any superfluous edge of~$H$.
Then $G-e$ has a cubic outer vertex that is an end of~$f$;
the other end of $f$ is either noncubic, or it is another cubic outer vertex.
\end{lem}
\begin{proof}
Recall definition~\ref{Def:expansion-index-two},
and observe the following. If each of $a_1,a_2,b_1,b_2$ has degree at least
three in~$G-f$, then $G-f$ may be obtained from the brace~$H-f$ by an expansion
of index two; whence $G-f$ is a brace
(by Proposition~\ref{prp:expansion-operations-preserve-brace-property});
contradiction. It follows that one of $a_1,a_2,b_1,b_2$ is a cubic
end of~$f$. Adjust notation so that $a_1$ is a cubic end of~$f$.
Clearly, if the other end of $f$ is not in $\{b_1,b_2\}$ then it is noncubic.

\smallskip
Now suppose that the other end of $f$ is in $\{b_1,b_2\}$; adjust notation so that $f:=a_1b_1$.
Assume that $b_1$ is noncubic.
Let $u$ denote the neighbor of $a_1$ that is distinct from $b_0$ and $b_1$.
See Figure~\ref{fig-1:index-two}.
Observe that the graph~$\widehat{G-f}$ may be obtained from the brace~$H-f$
by an expansion of index one;
in particular, $\widehat{G-f} := (H-f) \{ b \rightarrow  b_1a_0b_2 \} + a_0u$.
Thus $\widehat{G-f}$ is a brace
(by Proposition~\ref{prp:expansion-operations-preserve-brace-property}).
Consequently, $f$ is a strictly thin edge of index one,
contrary to our hypothesis that $G$ is devoid of such edges.
Thus $b_1$ is also cubic.
\end{proof}

\begin{lem}
\label{lem-2:index-two}
Let $f$ denote any superfluous edge of~$H$. Then, in~$G-e$,
either $f$ joins two outer vertices, or otherwise $f$
is adjacent with an edge that joins two outer vertices.
\end{lem}
\begin{proof}
By Lemma~\ref{lem-1:index-two}, $G-e$ has a cubic outer vertex that is an end of~$f$.
Adjust notation so that $a_1$ is a cubic end of~$f$.
If the other end of~$f$ lies in $\{b_1,b_2\}$ then we are done.

\begin{figure}[!htb]
\centering
\subfigure[brace~$H:=\widehat{G-e}$]
{
\begin{tikzpicture}

\draw (1.3,1.55)node[nodelabel]{$u'$};
\draw (0.25,1.5)node[nodelabel]{$u$};

\draw (0.5,1.5) -- (0.3,1.2);
\draw (0.5,1.5) -- (0.7,1.2);

\draw (1,1.5) -- (1,1.2);
\draw (1,1.5) -- (0.8,1.2);
\draw (1,1.5) -- (1.2,1.2);

\draw (1.5,-0.75) -- (0.5,0.5);
\draw (1.5,-0.75) -- (1,0.5);
\draw (1.5,-0.75) -- (2,0.5);
\draw (1.5,-0.75) -- (2.5,0.5);

\draw (1.5,2.75) -- (0.5,1.5);
\draw[ultra thick] (1.5,2.75) -- (1,1.5);
\draw (1.5,2.75) -- (2,1.5);
\draw (1.5,2.75) -- (2.5,1.5);

\draw (1.5,2.75)node{}node[nodelabel,above]{$a$};

\draw (1.5,-0.75)node[fill=black]{}node[nodelabel,below]{$b$};

\draw (0.5,1.5)node[fill=black]{};
\draw (1,1.5)node[fill=black]{};
\draw (2,1.5)node[fill=black]{};
\draw (2.5,1.5)node[fill=black]{};

\draw (0.5,0.5)node{};
\draw (1,0.5)node{};
\draw (2,0.5)node{};
\draw (2.5,0.5)node{};

\draw (0,0) -- (3,0) -- (3,2) -- (0,2) -- (0,0);

\draw (1.5,-2.05)node[nodelabel]{};
\end{tikzpicture}
%\label{fig:}
}
\hspace*{0.5in}
\subfigure[brace~$G$]
{
\begin{tikzpicture}

\draw (1.3,1.55)node[nodelabel]{$u'$};
\draw (0.25,1.5)node[nodelabel]{$u$};

\draw (0.5,1.5) -- (0.3,1.2);
\draw (0.5,1.5) -- (0.7,1.2);

\draw (1,1.5) -- (1,1.2);
\draw (1,1.5) -- (0.8,1.2);
\draw (1,1.5) -- (1.2,1.2);

\draw (1.5,3.5) to [out=0,in=90] (4,1) to [out=270,in=0] (1.5,-1.5);
\draw (4.2,1)node[nodelabel]{$e$};

\draw (1.5,-1.5) -- (1,-0.75);
\draw (1.5,-1.5) -- (2,-0.75);

\draw (1.5,3.5) -- (1,2.75);
\draw (1.5,3.5) -- (2,2.75);

\draw (1,2.75) -- (0.5,1.5);
\draw[ultra thick] (1,2.75) -- (1,1.5);
\draw (1.2,2.3)node[nodelabel]{$f$};
\draw (2,2.75) -- (2,1.5);
\draw (2,2.75) -- (2.5,1.5);

\draw (1,-0.75) -- (0.5,0.5);
\draw (1,-0.75) -- (1,0.5);
\draw (2,-0.75) -- (2,0.5);
\draw (2,-0.75) -- (2.5,0.5);

\draw (1,2.75)node{};
\draw (0.7,2.75)node[nodelabel]{$a_1$};
\draw (2,2.75)node{};
\draw (2.35,2.75)node[nodelabel]{$a_2$};

\draw (1,-0.75)node[fill=black]{};
\draw (0.7,-0.75)node[nodelabel]{$b_1$};
\draw (2,-0.75)node[fill=black]{};
\draw (2.35,-0.75)node[nodelabel]{$b_2$};

\draw (1.5,3.5)node[fill=black]{};
\draw (1.5,3.85)node[nodelabel]{$b_0$};

\draw (1.5,-1.5)node{};
\draw (1.5,-1.85)node[nodelabel]{$a_0$};

\draw (0.5,1.5)node[fill=black]{};
\draw (1,1.5)node[fill=black]{};
\draw (2,1.5)node[fill=black]{};
\draw (2.5,1.5)node[fill=black]{};

\draw (0.5,0.5)node{};
\draw (1,0.5)node{};
\draw (2,0.5)node{};
\draw (2.5,0.5)node{};

\draw (0,0) -- (3,0) -- (3,2) -- (0,2) -- (0,0);
\end{tikzpicture}
%\label{fig:}
}
\hspace*{0.5in}
\subfigure[brace~$\widehat{G-f}$]
{
\begin{tikzpicture}

\draw (0.5,1.5) -- (2,2.75);
\draw (0.5,1.5) to [out=180,in=90] (-0.75,0) to [out=270,in=180] (1.5,-1.5);
%\draw (0.5,1.5) to [out=,in=] () to [out=,in=] (

\draw (1.3,1.75)node[nodelabel]{$u'$};
\draw (0.25,1.7)node[nodelabel]{$u$};

\draw (0.5,1.5) -- (0.3,1.2);
\draw (0.5,1.5) -- (0.7,1.2);

\draw (1,1.5) -- (1,1.2);
\draw (1,1.5) -- (0.8,1.2);
\draw (1,1.5) -- (1.2,1.2);

%\draw (1.5,3.5) to [out=0,in=90] (4,1) to [out=270,in=0] (1.5,-1.5);
%\draw (4.2,1)node[nodelabel]{$e$};

\draw (1.5,-1.5) -- (1,-0.75);
\draw (1.5,-1.5) -- (2,-0.75);

%\draw (1.5,3.5) -- (1,2.75);
%\draw (1.5,3.5) -- (2,2.75);

%\draw (1,2.75) -- (0.5,1.5);
%\draw[ultra thick] (1,2.75) -- (1,1.5);
%\draw (1.2,2.3)node[nodelabel]{$f$};
\draw (2,2.75) -- (2,1.5);
\draw (2,2.75) -- (2.5,1.5);

\draw (1,-0.75) -- (0.5,0.5);
\draw (1,-0.75) -- (1,0.5);
\draw (2,-0.75) -- (2,0.5);
\draw (2,-0.75) -- (2.5,0.5);

%\draw (1,2.75)node{};
%\draw (0.7,2.75)node[nodelabel]{$a_1$};
\draw (2,2.75)node{};
\draw (2.35,2.75)node[nodelabel]{$a_2$};

\draw (1,-0.75)node[fill=black]{};
\draw (0.7,-0.75)node[nodelabel]{$b_1$};
\draw (2,-0.75)node[fill=black]{};
\draw (2.35,-0.75)node[nodelabel]{$b_2$};

%\draw (1.5,3.5)node[fill=black]{};
%\draw (1.5,3.85)node[nodelabel]{$b_0$};

\draw (1.5,-1.5)node{};
\draw (1.5,-1.85)node[nodelabel]{$a_0$};

\draw (0.5,1.5)node[fill=black]{};
\draw (1,1.5)node[fill=black]{};
\draw (2,1.5)node[fill=black]{};
\draw (2.5,1.5)node[fill=black]{};

\draw (0.5,0.5)node{};
\draw (1,0.5)node{};
\draw (2,0.5)node{};
\draw (2.5,0.5)node{};

\draw (0,0) -- (3,0) -- (3,2) -- (0,2) -- (0,0);
\end{tikzpicture}
%\label{fig:}
}
\caption{Illustration for the proof of Lemma~\ref{lem-2:index-two}}
\label{fig-2:index-two}
\end{figure}
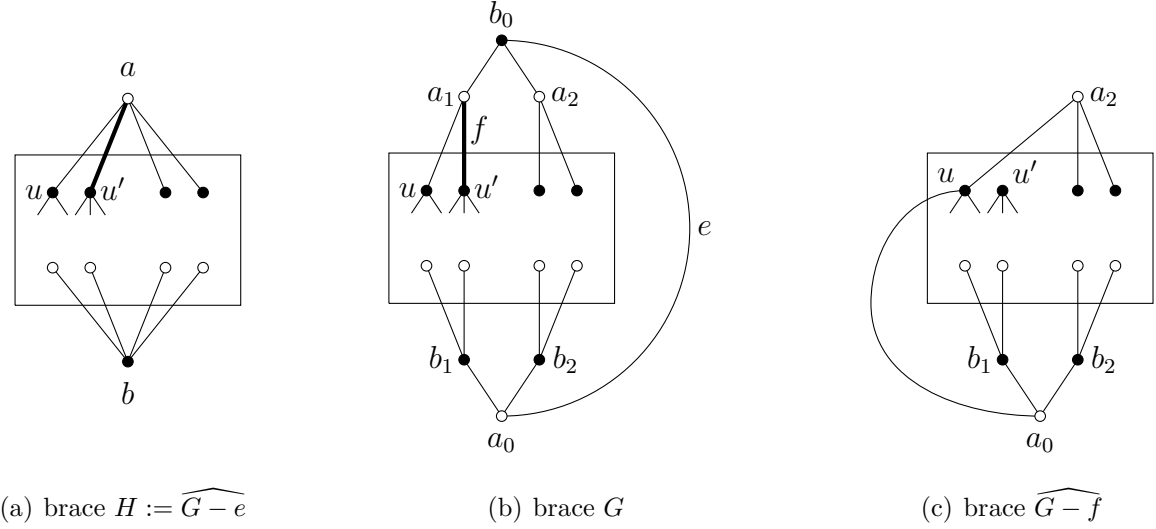

\smallskip
Now suppose that $f \notin \partial(\{b_1,b_2\})$,
and let $u'$ denote the noncubic end of~$f$.
Assume that $a_1b_1, a_1b_2 \notin E(G)$.
Let $u$ denote the neighbor of $a_1$ that is distinct from $b_0$ and $u'$.
Note that $u \notin \{b_1,b_2\}$.
See Figure~\ref{fig-2:index-two}.
Observe that the graph~$\widehat{G-f}$ may be obtained from the brace~$H-f$
by an expansion of index one;
in particular, $\widehat{G-f} := (H-f) \{b \rightarrow b_1a_0b_2 \} + a_0u$.
Thus $\widehat{G-f}$ is a brace
(by Proposition~\ref{prp:expansion-operations-preserve-brace-property}).
Consequently, $f$ is a strictly thin edge of index one,
contrary to our hypothesis that $G$ is devoid of such edges.
Thus, one of $a_1b_1$ and $a_1b_2$ is an edge of~$G$;
whence $f$ is adjacent with an edge that joins two outer vertices.
\end{proof}

\begin{cor}
\label{cor:index-two}
If $(e,F)$ is a minimality-preserving pair of~$G$, then~$G-e$ has two adjacent outer vertices $u$~and~$w$
such that $F \subset \partial(u) \cup \partial(w)$. Furthermore, if $F \cap \partial(u) \neq \emptyset$ then $u$
is cubic; likewise, if $F \cap \partial(w) \neq \emptyset$ then $w$ is cubic.
Consequently, $|F| \in \{1,2,3\}$. \qed
\end{cor}

\subsection{Main Theorem}
\label{sec:main-theorem}

We are now ready to state and prove the Main Theorem.

\begin{thm}
\label{thm:narrow-minimality-preserving-pair-in-minimal-braces}
{\sc [Main Theorem]}
Every minimal brace $G \notin \mathcal{G}$
has a minimality-preserving pair $(e,F)$ such that either $F = \emptyset$,
or otherwise $F$ satisfies the following properties in~$G-e$:
\begin{enumerate}[(i)]
\item If ${\rm index}(e)=1$ then $G-e$ has an outer vertex, say~$v$, such that $F \subset \partial(v)$.
\item If ${\rm index}(e)=2$ then $G-e$ has two adjacent outer vertices, say $u$~and~$w$,
such that $F \subset \partial(u) \cup \partial(w)$.
\item For each $f \in F$, an end of~$f$ is cubic if and only if it is an outer vertex.
\end{enumerate}
Consequently, $|F| \leq {\rm index}(e)+1$. Furthermore, if ${\rm index}(e)=1$ and $|F|=2$
then $G$ is isomorphic to a stable-extension of the minimal brace~$J:=\widehat{G-e}-F$.
\end{thm}
\begin{proof}
By McCuaig's Theorem and its corollary (\ref{cor:minimality-preserving-pair-in-minimal-braces}),
$G$ has a strictly thin edge (of index one or two); furthermore, for any such edge~$e$,
there exists a minimality-preserving pair~$(e,F)$.

\smallskip
If $G$ has a strictly thin edge of index one, say~$e$, then we choose any minimality-preserving pair $(e,F)$,
and we are done by invoking Lemma~\ref{lem-1:index-one}, Corollary~\ref{cor:index-one}
and Lemma~\ref{lem-3:index-one}.

\smallskip
Now suppose that $G$ is devoid of strictly thin edges of index one.
Thus $G$ has a strictly thin edge of index two, say~$e$.
We choose any minimality-preserving pair $(e,F)$,
and we are done by invoking Lemma~\ref{lem-1:index-two} and Corollary~\ref{cor:index-two}.
\end{proof}

For convenience, let us introduce the following definition.

\begin{Def}
\label{Def:narrow-minimality-preserving-pair}
{\sc [Narrow Minimality-Preserving Pair]}
A minimality-preserving pair $(e,F)$ of a minimal brace~$G$ is {\it narrow}
if it satisfies the statement of the Main Theorem~(\ref{thm:narrow-minimality-preserving-pair-in-minimal-braces}).
\end{Def}

We may thus condense the statement of the
Main Theorem as follows:
every minimal brace~$G \notin \mathcal{G}$
has a narrow minimality-preserving pair.

\smallskip
Recall that we presented two equivalent versions of McCuaig's Theorem:
first a `reduction version' (\ref{thm:strictly-thin-edge-in-simple-braces}),
and second a `generation version' (\ref{thm:generating-simple-braces}).
In the same spirit, Theorem~\ref{thm:narrow-minimality-preserving-pair-in-minimal-braces}
(as stated) may be viewed as a `reduction version' for minimal braces.
The interested reader may obtain a `generation version'
(equivalent to Theorem~\ref{thm:narrow-minimality-preserving-pair-in-minimal-braces})
for minimal braces
by defining `extension operations' that could be potentially useful in obtaining a (larger) minimal brace
from a (smaller) minimal brace.
Each such extension operation may be viewed as a
sequence of index zero expansions (possibly none) followed by an index one or index two
expansion with some additional restrictions
(depending on the cardinality and structure of the set~$F$ and the index of the strictly thin edge~$e$).
This is in fact the viewpoint adopted by Norine and Thomas~\cite{noth06}
in their paper on minimal bricks.

\smallskip
We shall not define all of the aforementioned extension operations here (except for the one
we have already described: Definition~\ref{Def:stable-extension}). The reader may verify
that if a graph~$G$ is isomorphic to a stable-extension of a simple brace~$J$, then $G$ may be obtained
from~$J$ by first adding two (adjacent) edges and then performing
an index one expansion.
This observation, coupled with
Proposition~\ref{prp:expansion-operations-preserve-brace-property}, yields the following.

\begin{cor}
\label{cor:stable-extension-produces-simple-brace}
Any graph~$G$, that is isomorphic to a stable-extension of a simple brace~$J$,
is also a simple brace. \qed
\end{cor}

\begin{comment}

Here, we shall not describe all of these `extension operations'. We describe
only one of them that is useful to us in the remainder of this paper.

\smallskip
The following is easily verified.

\begin{prp}
\label{prp:the-case-of-S-extension}
Let $(e,F)$ denote a narrow minimality-preserving pair of a minimal brace~$G$
such that ${\rm index}(e)=1$ and $|F|=2$, and let $J:=\widehat{G-e}-F$.
Let $f_1$ and $f_2$ denote the distinct members of $F$, and let $b_1$ and $b_2$
denote their noncubic ends, respectively. Let $a_1$ denote the noncubic
end of $e$. Let $a_2$ denote the noncubic outer vertex of $G-e$.
Then $G$ is isomorphic to the $S$-extension of $J$,
where $S:=\{a_1,a_2,b_1,b_2\}$. \qed
\end{prp}

Observe that an $S$-extension operation (applied to a simple brace)
may be viewed as a sequence of three expansion operations
(with some additional restrictions):
two operations of index zero followed by an operation of index one.
The following is an immediate consequence of
Proposition~\ref{prp:expansion-operations-preserve-brace-property}.

It should however be noted that, in the statement of the above corollary,
even if $J$ is a minimal brace, the simple brace~$G$ may or may not be minimal.
\end{comment}

%\newpage
\section{An application}
\label{sec:application}

In this section, as an application of the
Main Theorem~(\ref{thm:narrow-minimality-preserving-pair-in-minimal-braces}),
we will prove that if $G$ is a minimal brace (where $n_G \geq 6$) then
$m_G \leq 5 \cdot n_G - 10$;
furthermore, we shall provide a complete characterization of minimal braces that
meet this upper bound.

\smallskip
We begin by defining an infinite family~$\mathcal{Q}$ of minimal braces,
each of whose members meets this upper bound;
its smallest member is $Q_{12}$.
For $n \geq 7$, we define the graph~$Q_{2n}$ as the $S$-extension of $Q_{2n-2}$ ---
where the set~$S$ comprises the noncubic vertices of $Q_{2n-2}$.
Now let $\mathcal{Q}:=\{Q_{2n} : n \geq 6\}$.
Since~$Q_{12}$ is a simple brace, it follows from
Corollary~\ref{cor:stable-extension-produces-simple-brace}
that each member of~$\mathcal{Q}$ is a simple brace.
Furthermore, by Proposition~\ref{prp:brace-noncubic-stable-implies-minimal},
we infer that each member of $\mathcal{Q}$ is in fact a minimal brace.
By summing degrees, observe that if $G \in \mathcal{Q}$ then $m_G = 5 \cdot n_G - 10$.

\smallskip
We now prove the following easy corollary of the Main Theorem.
\begin{cor}
\label{cor:edge-bound-preservation}
Let $(e,F)$ denote a narrow minimality-preserving pair of a minimal brace~$G$,
and let $J:=\widehat{G-e}-F$.
Assume that $m_J \leq 5 \cdot n_J - 10$. Then the following hold:
\begin{enumerate}[(i)]
\item $m_G \leq 5 \cdot n_G-10$.
\item If $m_G = 5 \cdot n_G-10$ then $m_J = 5 \cdot n_J - 10$, ${\rm index}(e)=1$ and $|F|=2$;
consequently, $G$ is isomorphic to a stable-extension of~$J$.
\end{enumerate}
\end{cor}
\begin{proof}
By Proposition~\ref{prp:compare-order-size-of-G-and-J},
$n_J= n_G- {\rm index}(e)$ and $m_G = m_J + 1 + 2 \cdot {\rm index}(e) + |F|$.
By Theorem~\ref{thm:narrow-minimality-preserving-pair-in-minimal-braces}, $|F| \leq {\rm index}(e) + 1$.
We now consider two cases depending on the index of $e$.

\smallskip
First consider the case: ${\rm index}(e)=2$. Using the equations and inequalities noted above, we have:
$m_G = m_J + 5 + |F| < (5 \cdot n_J - 10) + 5 + 5 = 5 \cdot n_J = 5 \cdot n_G - 10.$
Thus, in this case, the strict inequality $m_G < 5 \cdot n_G - 10$ holds.

\smallskip
Now consider the case: ${\rm index}(e)=1$. Using the same equations and inequalities as before, we have:
$m_G = m_J + 3 + |F| \leq (5 \cdot n_J - 10) + 3 + 2 = 5 \cdot n_J - 5 = 5 \cdot n_G - 10$.
Hence, we have the inequality $m_G \leq 5 \cdot n_G - 10$; equality holds
if and only if $m_J = 5 \cdot n_J -10$ and $|F|=2$.
By the last part of Theorem~\ref{thm:narrow-minimality-preserving-pair-in-minimal-braces},
we infer that $G$
is isomorphic to a stable-extension of~$J$.
This completes the proof of Corollary~\ref{cor:edge-bound-preservation}.
\end{proof}

We are now ready to prove the main result of this section.

\begin{thm}
\label{thm:extremal-minimal-braces}
Let $G$ denote a minimal brace that is not in $\{K_2, C_4, K_{3,3}, B_8, B_{10}, Q_{10}^+\}$.
Then $m_G \leq 5 \cdot n_G - 10$, and
equality holds if and only if $G \in \{M_{10}, B_{12}\} \cup \mathcal{Q}$.
\end{thm}
\begin{proof}
We proceed by induction on the order of the graph.
Let $G[A,B]$ denote a minimal brace that is not in $\{K_2,C_4, K_{3,3}, B_8, B_{10}, Q_{10}^+\}$.

\smallskip
The reader may verify, using
Propositions \ref{prp:minimal-braces-order-at-most-eight}~and~\ref{prp:minimal-braces-order-ten},
that the desired conclusion holds when $n_G \leq 5$.
Now suppose that $n_G \geq 6$.
If $G$ is cubic then $m_G = 3 \cdot n_G < 5 \cdot n_G -10$.
We may thus assume that $G$ is noncubic.
Now one may verify, using Proposition~\ref{prp:minimal-braces-order-twelve-size-twenty},
that the desired conclusion holds when $n_G = 6$.
Henceforth, suppose that $n_G \geq 7$.
If $G$ is a biwheel then $m_G = 4 \cdot n_G - 4 < 5 \cdot n_G-10$.
We may thus assume that $G$ is not a biwheel.
Consequently, $G \notin \mathcal{G}$.

\smallskip
By the Main Theorem~(\ref{thm:narrow-minimality-preserving-pair-in-minimal-braces}),
$G$ has a narrow minimality-preserving pair, say $(e,F)$.
Let $J$ denote the minimal brace~$\widehat{G-e}-F$.
Note that $n_J = n_G - {\rm index}(e)$ where ${\rm index}(e) \in \{1,2\}$;
whence $n_J \geq 5$.
By invoking the induction hypothesis, either $J \in \{B_{10}, Q_{10}^+\}$,
or otherwise $m_J \leq 5 \cdot n_J-10$ and
equality holds if and only if $J \in \{M_{10}, B_{12}\} \cup \mathcal{Q}$.

\smallskip
First suppose that $J \in \{B_{10}, Q_{10}^+\}$.
In particular, $n_J=5$ and $m_J=16$. Consequently, ${\rm index}(e)=2$ and $n_G = 7$.
By Proposition~\ref{prp:compare-order-size-of-G-and-J},
$m_G = m_J + 1 + 2 \cdot {\rm index}(e) + |F|$
where $|F| \leq 3$.
It follows that $m_G \leq 24 < 25 = 5 \cdot n_G - 10$.
The desired conclusion holds.

\smallskip
Now suppose that $m_J \leq 5 \cdot n_J - 10$ and
equality holds if and only if $J \in \{M_{10}, B_{12}\} \cup \mathcal{Q}$.
By invoking Corollary~\ref{cor:edge-bound-preservation},
we infer that $m_G \leq 5 \cdot n_G - 10$.
Since $n_G \geq 7$,
it only remains to prove the following: if $m_G = 5 \cdot n_G - 10$
then $G \in \mathcal{Q}$.

\smallskip
Henceforth, assume that $m_G = 5 \cdot n_G - 10$.
It now follows from Corollary~\ref{cor:edge-bound-preservation}
that $m_J = 5 \cdot n_J - 10$. Thus $J \in \{M_{10},B_{12} \} \cup \mathcal{Q}$.
Furthermore, $J$ has a stable set $S$ that
meets each color class in precisely two vertices, and $G$ is isomorphic to the
$S$-extension of~$J$.
We let $A'$ and $B'$ denote the color classes of $J$, and we let $a_0$ and $b_0$
denote the extension vertices of~$G$. Thus $A=A' \cup \{a_0\}$ and $B=B' \cup \{b_0\}$.

\smallskip
One may easily verify that $M_{10}$, shown
in Figure~\ref{fig:M_10},
has no stable set that meets each color class in two vertices.
Thus, either $J$ is isomorphic to $B_{12}$
or otherwise~$J \in \mathcal{Q}$.

\begin{figure}[!htb]
\centering
\subfigure[$B_{12}$]
{
\begin{tikzpicture}[scale=1]

\draw (180:4) to [out=90,in=180] (90:3) to [out=0,in=90] (36:2);

\draw (180:4) to [out=270,in=180] (270:3) to [out=0,in=270] (324:2);

\draw (180:4) -- (180:2);
\draw (180:4) to [out=60,in=180] (108:2);
\draw (180:4) to [out=300,in=180] (252:2);

\draw (0:0) -- (0:2);
\draw (0:0) -- (72:2);
\draw (0:0) -- (144:2);
\draw (0:0) -- (216:2);
\draw (0:0) -- (288:2);

\draw (0,0) circle(2);

\draw (0:2)node[fill=black]{};
\draw (72:2)node[fill=black]{};
\draw (144:2)node[fill=black]{}node[left,nodelabel]{$b_1$};
\draw (216:2)node[fill=black]{}node[left,nodelabel]{$b_2$};
\draw (288:2)node[fill=black]{};
\draw (180:4)node[fill=black]{};
\draw (180:4.2)node[nodelabel]{$b$};

\draw (36:2)node{}node[right,nodelabel]{$a_1$};
\draw (108:2)node{};
\draw (180:2)node{};
\draw (252:2)node{};
\draw (324:2)node{}node[right,nodelabel]{$a_2$};
\draw (0:0)node{};
\draw (72:-0.3)node[nodelabel]{$a$};

\end{tikzpicture}
\label{fig:B_12}
}
\hspace*{0.2in}
\subfigure[$S$-extension of $B_{12}$ where \mbox{$S:=\{a_1,a_2,b_1,b_2\}$}]
{
\begin{tikzpicture}[scale=1]

\draw[thick] (180:1.1) to [out=30,in=180] (30:1);

\draw[thick] (36:2) -- (30:1) -- (324:2);
\draw[thick] (216:2) -- (180:1.1) -- (144:2);

\draw (180:4) to [out=90,in=180] (90:3) to [out=0,in=90] (36:2);

\draw (180:4) to [out=270,in=180] (270:3) to [out=0,in=270] (324:2);

\draw (180:4) -- (180:2);
\draw (180:4) to [out=60,in=180] (108:2);
\draw (180:4) to [out=300,in=180] (252:2);

\draw (0:0) -- (0:2);
\draw (0:0) -- (72:2);
\draw (0:0) -- (144:2);
\draw (0:0) -- (216:2);
\draw (0:0) -- (288:2);

\draw (0,0) circle(2);

\draw (0:2)node[fill=black]{};
\draw (72:2)node[fill=black]{};
\draw (144:2)node[fill=black]{}node[left,nodelabel]{$b_1$};
\draw (216:2)node[fill=black]{}node[left,nodelabel]{$b_2$};
\draw (288:2)node[fill=black]{};
\draw (180:4)node[fill=black]{};
\draw (180:4.2)node[nodelabel]{$b$};

\draw (36:2)node{}node[right,nodelabel]{$a_1$};
\draw (108:2)node{};
\draw (180:2)node{};
\draw (252:2)node{};
\draw (324:2)node{}node[right,nodelabel]{$a_2$};
\draw (0:0)node{};
\draw (72:-0.3)node[nodelabel]{$a$};

\draw (180:1.1)node{};
\draw (180:1.4)node[nodelabel]{$a_0$};
\draw (30:1)node[fill=black]{};
\draw (20:1.4)node[nodelabel]{$b_0$};

\draw (90:2.7)node[nodelabel]{$f$};

\end{tikzpicture}
\label{fig:S-ext-of-B_12}
}
\caption{The $S$-extension of $B_{12}$ is not a minimal brace}
%\label{fig:}
\end{figure}
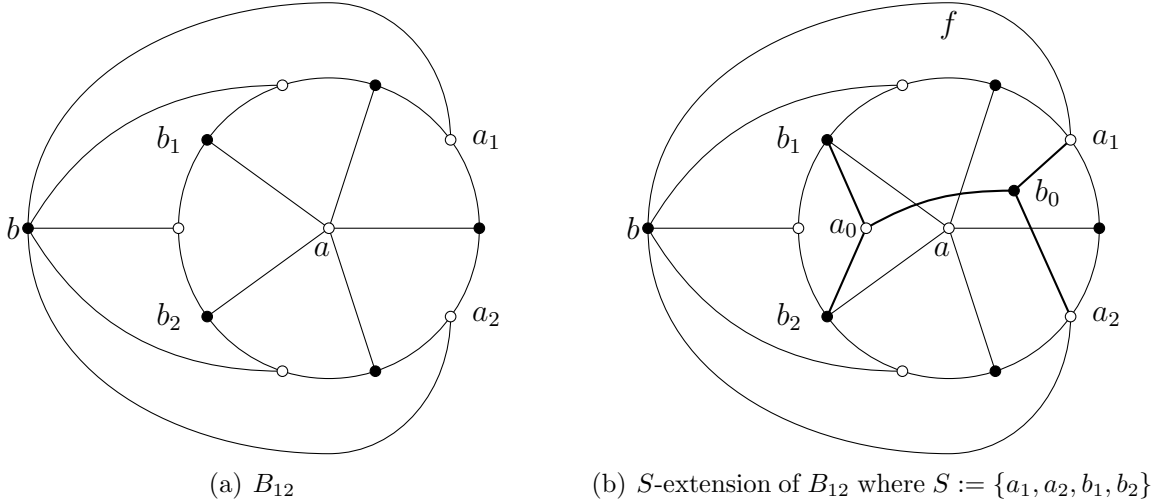

\smallskip
First consider the case in which $J$ is isomorphic to $B_{12}$.
We let $a \in A'$ and $b \in B'$ denote the noncubic vertices of~$J$.
The reader may easily verify that, up to symmetry, $B_{12}$ has only one
stable set~$\{a_1,b_1,a_2,b_2\}$
that meets each color class in precisely two vertices --- as shown in Figure~\ref{fig:B_12}.
Thus, we let $S:=\{a_1,b_1,a_2,b_2\}$;
whence $G$ is isomorphic to the graph shown in Figure~\ref{fig:S-ext-of-B_12}.
We shall now arrive at a contradiction by showing that $G-f$ is a brace --- where $f:=a_1b$.
By Proposition~\ref{prp:brace-characterization}, for each nonempty set~$Z \subset A$
such that $|Z| < |A|-1$, the inequality $|N_G(Z)| \geq |Z|+2$ holds;
furthermore, it suffices to verify that $|N_{G-f}(Z)| \geq |Z|+2$ holds;
by comparing the graphs $G$ and $G-f$, observe that we only
need to check those sets $Z$ that satisfy the following: $\{a_1\} \subseteq Z \subseteq \{a,a_0,a_1\}$.
The reader may verify that, for each set~$Z$ that satisfies $\{a_1\} \subseteq Z \subseteq \{a,a_0,a_1\}$,
the inequality $|N_{G-f}(Z)| \geq |Z|+2$ holds. Thus $f$ is a superfluous edge in~$G$;
contradiction. (In fact, by symmetry, any edge of~$G$, both of whose ends are noncubic, is superfluous.)

\smallskip
Now consider the case in which $J \in \mathcal{Q}$.
We let $A':=\{a_1,a_2,w_1,w_2,\dots,w_{n_J-2}\}$
and $B':=\{b_1,b_2,u_1,u_2,\dots, u_{n_J-2} \}$
--- such that $T:=\{a_1,a_2,b_1,b_2\}$ is the set of noncubic vertices of~$J$,
and $w_iu_i \in E(J)$ for each $i \in \{1, 2, \dots, n_J-2\}$ ---
as shown in Figure~\ref{fig:Q_14}.
We consider two subcases depending on whether the set~$T$ intersects with
the stable set~$S$ or not.

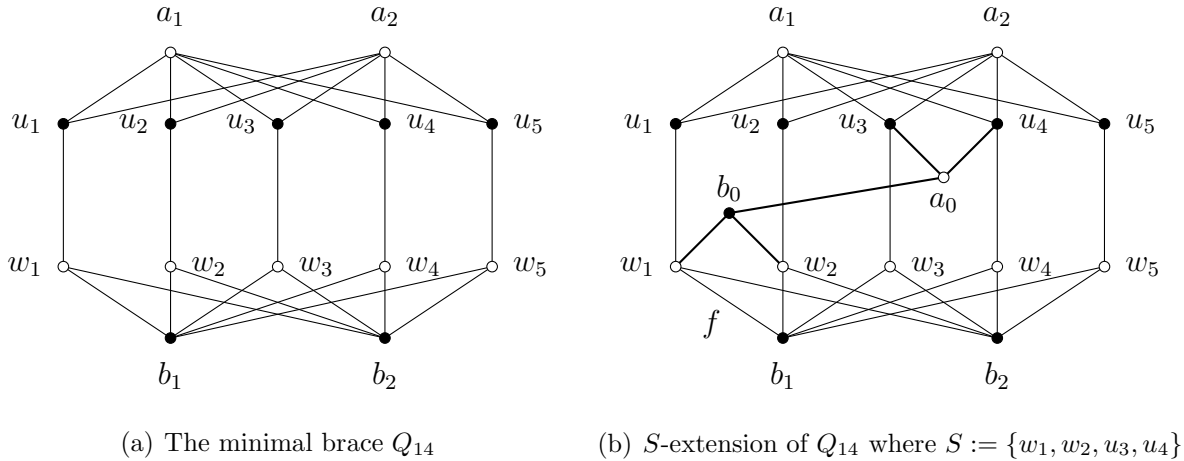
\begin{figure}[!htb]
\centering
\subfigure[The minimal brace~$Q_{14}$]
{
\begin{tikzpicture}[scale=0.95]

\draw (4.5,-1) -- (0,0);
\draw (4.5,-1) -- (1.5,0);
\draw (4.5,-1) -- (3,0);
\draw (4.5,-1) -- (4.5,0);
\draw (4.5,-1) -- (6,0);

\draw (1.5,-1) -- (0,0);
\draw (1.5,-1) -- (1.5,0);
\draw (1.5,-1) -- (3,0);
\draw (1.5,-1) -- (4.5,0);
\draw (1.5,-1) -- (6,0);

\draw (4.5,3) -- (0,2);
\draw (4.5,3) -- (1.5,2);
\draw (4.5,3) -- (3,2);
\draw (4.5,3) -- (4.5,2);
\draw (4.5,3) -- (6,2);

\draw (1.5,3) -- (0,2);
\draw (1.5,3) -- (1.5,2);
\draw (1.5,3) -- (3,2);
\draw (1.5,3) -- (4.5,2);
\draw (1.5,3) -- (6,2);

\draw (0,0) -- (0,2);
\draw (1.5,0) -- (1.5,2);
\draw (3,0) -- (3,2);
\draw (4.5,0) -- (4.5,2);
\draw (6,0) -- (6,2);

\draw (1.5,3)node{}node[above,nodelabel]{$a_1$};
\draw (4.5,3)node{}node[above,nodelabel]{$a_2$};
\draw (0,0)node{}node[left,nodelabel]{$w_1$};
\draw (1.5,0)node{}node[right,nodelabel]{$w_2$};
\draw (3,0)node{}node[right,nodelabel]{$w_3$};
\draw (4.5,0)node{}node[right,nodelabel]{$w_4$};
\draw (6,0)node{}node[right,nodelabel]{$w_5$};

\draw (1.5,-1)node[fill=black]{}node[below,nodelabel]{$b_1$};
\draw (4.5,-1)node[fill=black]{}node[below,nodelabel]{$b_2$};
\draw (0,2)node[fill=black]{}node[left,nodelabel]{$u_1$};
\draw (1.5,2)node[fill=black]{}node[left,nodelabel]{$u_2$};
\draw (3,2)node[fill=black]{}node[left,nodelabel]{$u_3$};
\draw (4.5,2)node[fill=black]{}node[right,nodelabel]{$u_4$};
\draw (6,2)node[fill=black]{}node[right,nodelabel]{$u_5$};

\end{tikzpicture}
\label{fig:Q_14}
}
\subfigure[$S$-extension of $Q_{14}$ where $S:=\{w_1,w_2,u_3,u_4\}$]
{
\begin{tikzpicture}[scale=0.95]

\draw[thick] (0.75,0.75) -- (3.75,1.25);

\draw[thick] (0,0) -- (0.75,0.75) -- (1.5,0);
\draw[thick] (3,2) -- (3.75,1.25) -- (4.5,2);

\draw (4.5,-1) -- (0,0);
\draw (4.5,-1) -- (1.5,0);
\draw (4.5,-1) -- (3,0);
\draw (4.5,-1) -- (4.5,0);
\draw (4.5,-1) -- (6,0);

\draw (1.5,-1) -- (0,0);
\draw (1.5,-1) -- (1.5,0);
\draw (1.5,-1) -- (3,0);
\draw (1.5,-1) -- (4.5,0);
\draw (1.5,-1) -- (6,0);

\draw (4.5,3) -- (0,2);
\draw (4.5,3) -- (1.5,2);
\draw (4.5,3) -- (3,2);
\draw (4.5,3) -- (4.5,2);
\draw (4.5,3) -- (6,2);

\draw (1.5,3) -- (0,2);
\draw (1.5,3) -- (1.5,2);
\draw (1.5,3) -- (3,2);
\draw (1.5,3) -- (4.5,2);
\draw (1.5,3) -- (6,2);

\draw (0,0) -- (0,2);
\draw (1.5,0) -- (1.5,2);
\draw (3,0) -- (3,2);
\draw (4.5,0) -- (4.5,2);
\draw (6,0) -- (6,2);

\draw (1.5,3)node{}node[above,nodelabel]{$a_1$};
\draw (4.5,3)node{}node[above,nodelabel]{$a_2$};
\draw (0,0)node{}node[left,nodelabel]{$w_1$};
\draw (1.5,0)node{}node[right,nodelabel]{$w_2$};
\draw (3,0)node{}node[right,nodelabel]{$w_3$};
\draw (4.5,0)node{}node[right,nodelabel]{$w_4$};
\draw (6,0)node{}node[right,nodelabel]{$w_5$};

\draw (1.5,-1)node[fill=black]{}node[below,nodelabel]{$b_1$};
\draw (4.5,-1)node[fill=black]{}node[below,nodelabel]{$b_2$};
\draw (0,2)node[fill=black]{}node[left,nodelabel]{$u_1$};
\draw (1.5,2)node[fill=black]{}node[left,nodelabel]{$u_2$};
\draw (3,2)node[fill=black]{}node[left,nodelabel]{$u_3$};
\draw (4.5,2)node[fill=black]{}node[right,nodelabel]{$u_4$};
\draw (6,2)node[fill=black]{}node[right,nodelabel]{$u_5$};

\draw (0.75,0.75)node[fill=black]{};
\draw (0.75,1.1)node[nodelabel]{$b_0$};
\draw (3.75,1.25)node{};
\draw (3.75,0.9)node[nodelabel]{$a_0$};

\draw (0.5,-0.8)node[nodelabel]{$f$};

\end{tikzpicture}
\label{fig:S-extension-of-Q_14}
}
\caption{An $S$-extension of $Q_{14}$ that is not a minimal brace}
\end{figure}

\smallskip
First suppose that $S \cap T = \emptyset$. Thus, we may adjust notation so that
$S:= \{w_1,w_2,u_3,u_4\}$; see Figure~\ref{fig:S-extension-of-Q_14}.
Similarly to an earlier case, we will arrive at a contradiction by
showing that $G-f$ is a brace --- where $f:=b_1w_1$.
By Proposition~\ref{prp:brace-characterization}, for each nonempty set~$Z \subset A$
such that $|Z| < |A|-1$, the inequality $|N_G(Z)| \geq |Z|+2$ holds;
furthermore, it suffices to verify that $|N_{G-f}(Z)| \geq |Z|+2$ holds;
by comparing the graphs $G$ and $G-f$, observe that we only need to check
those sets $Z$ that satisfy the following: $\{w_1\} \subseteq Z \subseteq \{w_1,a_0,a_1,a_2\}$.
The reader may verify that, for each set $Z$ that satisfies
$\{w_1\} \subseteq Z \subseteq \{w_1,a_0,a_1,a_2\}$, the
inequality $|N_{G-f}(Z)| \geq |Z|+2$ holds. Thus $f$ is a superfluous edge in~$G$;
contradiction. (In fact, by symmetry, any edge of~$G$, both of whose ends are noncubic,
is superfluous.)

\smallskip
Now suppose that $S \cap T \neq \emptyset$. Observe that, in this case, $S = T$.
Consequently, $G \in \mathcal{Q}$, by definition of the family~$\mathcal{Q}$.
This completes the proof of Theorem~\ref{thm:extremal-minimal-braces}.
\end{proof}

\bigskip
\noindent
{\bf Acknowledgements:}
The first and third authors received support from Fundect-MS and from CNPq (Brazil).
The second author received support from
FAPESP (2018/04679-1) of Sao Paulo (Brazil) and
from the Austrian Science Foundation FWF (START grant Y463).
The authors are especially grateful to one of the anonymous referees who
helped in improving the presentation.

\bibliographystyle{plain}
\bibliography{clm}

\end{document}